\documentclass[10pt]{amsart}
\usepackage{amssymb, amsmath, amsthm, amsfonts}
\usepackage{mathrsfs,comment}
\usepackage{graphicx}
\usepackage{todonotes}
\usepackage{dsfont}
\usepackage[bookmarks=true]{hyperref}  
\usepackage{url}
\usepackage{enumerate}
\usepackage{tikz}
\usepackage{tikz-cd} 
\usepackage{eucal}
\usepackage[all]{xy}
\hypersetup{%
  bookmarksnumbered=true,%%
  colorlinks=true,%
  linkcolor=blue,%
  citecolor=blue,%
  filecolor=blue,%
  menucolor=blue,%
  urlcolor=blue,%
  pdfnewwindow=true,%
  pdfstartview=FitBH}

\newcommand{\R}{\mathbb{R}}
\newcommand{\C}{\mathbb{C}}
\newcommand{\Z}{\mathbb{Z}}
\newcommand{\F}{\mathbb{F}}

\newcommand{\E}{\mathcal{E}}
\newcommand{\I}{\mathcal{I}}
\newcommand{\J}{\mathcal{J}}

\DeclareMathOperator{\colim}{\mathrm{colim}}
\DeclareMathOperator{\coker}{\mathrm{coker}}

\DeclareMathOperator{\holim}{\mathrm{holim}}

\newcommand{\Ho}{\mathrm{Ho}}
\DeclareMathOperator{\homog}{--homog--}

\DeclareMathOperator{\poly}{--poly--}

\newcommand{\s}{\mathsf{Sp}}
\DeclareMathOperator{\T}{\mathsf{Top_\ast}}

\newcommand{\Aut}{\mathrm{Aut}}

\DeclareMathOperator{\ind}{\mathrm{ind}}
\DeclareMathOperator{\res}{\mathrm{res}}

\newcommand{\U}{\mathrm{U}}
\newcommand{\BU}{\mathrm{BU}}
\renewcommand{\O}{\mathrm{O}}
\newcommand{\BO}{\mathrm{BO}}

\theoremstyle{definition}
\newtheorem{thm}{Theorem}[section]

\newtheorem{prop}[thm]{Proposition}

\newtheorem{lem}[thm]{Lemma}
\newtheorem{cor}[thm]{Corollary}
\newtheorem{example}[thm]{Example}
\newtheorem{definition}[thm]{Definition}
\newtheorem{rem}[thm]{Remark}

\makeatletter
\newcommand*{\centerfloat}{%
  \parindent \z@
  \leftskip \z@ \@plus 1fil \@minus \textwidth
  \rightskip\leftskip
  \parfillskip \z@skip}
\makeatother

  \newcommand{\adjunction}[4]{
\xymatrix{
#1:#2 \ar@<0.7ex>[r] &
\ar@<0.7ex>[l] #3:#4
}}

\newcommand{\vsim}{\rotatebox[origin=c]{-90}{$\sim$}}
\newcommand{\dsim}{\rotatebox[origin=c]{-45}{$\sim$}}

\begin{document}

\title{Comparing the orthogonal and unitary functor calculi}
\author{Niall Taggart}
\email{ntaggart@mpim-bonn.mpg.de}
\address{Max Planck Institute for Mathematics, Vivatsgasse 7, 53111 Bonn, Germany}
%\date{}

\begin{abstract}
The orthogonal and unitary calculi give a method to study functors from the category of real or complex inner product spaces to the category of based topological spaces. We construct functors between the calculi from the complexification-realification adjunction between real and complex inner product spaces. These allow for movement between the versions of calculi, and comparisons between the Taylor towers produced by both calculi. We show that when the inputted orthogonal functor is weakly polynomial, the Taylor tower of the functor restricted through realification and the restricted Taylor tower of the functor agree up to weak equivalence. We further lift the homotopy level comparison of the towers to a commutative diagram of Quillen functors relating the model categories for orthogonal calculus and the model categories for unitary calculus. 
\end{abstract}

\maketitle

%-----------------------------------------------------------------------
%-----------------------------------------------------------------------
\section{Introduction}
%-----------------------------------------------------------------------
%-----------------------------------------------------------------------

The orthogonal and unitary calculi allow for the systematic study of functors from either the category of real inner product spaces, or the category of complex inner product spaces, to the category of based topological spaces. The motivating examples are $\BO(-) \colon V \longmapsto \BO(V)$, where $\BO(V)$ is the classifying space of the orthogonal group of $V$, and $\BU(-)\colon W \longmapsto \BU(W)$ where $\BU(W)$ is the classifying space of the unitary group of $W$. The foundations of orthogonal calculus were originally developed by Weiss in \cite{We95}, and later converted to a model category theoretic framework by Barnes and Oman in \cite{BO13}. From the unitary calculus perspective, it has long been known to the experts, with the foundations and model category framework developed by the author in \cite{Ta19}. 

In this paper we use the similarity between real and complex vector spaces, namely the complexification-realification adjunction to give a formal comparison of the calculi both on the homotopy level by comparing the towers, and on the model category level by constructing Quillen functors between the model categories for orthogonal calculus and the model categories for unitary calculus. In particular, this paper will allow for a more methodical way of using the calculi together, and to transfer calculations between them. 

We cover the basic background of the calculi in Section \ref{section: the calculi}. With this background in place, we begin with a comparison of the input functors in Section \ref{section: input functors}. In particular we construct two Quillen adjunctions between the input categories with precomposition with realification and precomposition with complexification respectively being right Quillen functors. 

Denote by $\E_0^\mathbf{O}$ the category of input functors for orthogonal calculus, that is, $\T$-enriched functors from the category of real inner product spaces to the category of based spaces, and denote by $\E_0^\mathbf{U}$ the unitary calculus analogue. For a full definition of these categories, see Definition \ref{def: input functors}. The realification of complex vector spaces induces a functor $r^*\colon \E_0^\mathbf{O} \to \E_0^\mathbf{U}$, and the complexification of real vector spaces induces a functor $c^*\colon \E_0^\mathbf{U} \to \E_0^\mathbf{O}$, full constructions of such are given in Section \ref{section: input functors}. These functors behave well with respect to homogeneous and polynomial functors, see Section \ref{section: poly and homog}, where we prove the following as Lemma \ref{homog r} and Lemma \ref{homog for c}. 

\begin{lem}
\hspace{10cm}
\begin{enumerate}
\item If an orthogonal functor $F$ is $n$-homogeneous, then $r^*F$ is $n$-homogeneous, where $r^*\colon \E_0^\mathbf{O} \to \E_0^\mathbf{U}$ is precomposition with the realification functor. 
\item If a unitary functor $F$ is $n$-homogeneous, then $c^*F$ is $(2n)$-homogeneous, where $c^*\colon \E_0^\mathbf{U} \to \E_0^\mathbf{O}$ is precomposition with the complexification functor. 
\end{enumerate}
\end{lem}

Utilising the Taylor tower for an inputted functor $F$, and the above result on homogeneous functors, we prove the following. This result appears as Theorem \ref{poly for r} and Theorem \ref{n-poly for c} in the text. This result only applies to reduced functors, that is, those with trivial $0$--polynomial approximation.

\begin{thm}
\hspace{10cm}
\begin{enumerate}
\item If a reduced orthogonal functor $F$ is $n$-polynomial, then $r^*F$ is $n$-polynomial, where $r^*\colon \E_0^\mathbf{O} \to \E_0^\mathbf{U}$ is precomposition with the realification functor. 
\item If a reduced unitary functor $F$ is $n$-polynomial, then $c^*F$ is $(2n)$-polynomial, where $c^*\colon \E_0^\mathbf{U} \to \E_0^\mathbf{O}$ is precomposition with the complexification functor. 
\end{enumerate}
\end{thm}

In Section \ref{section: poly and homog} we also construct Quillen adjunctions between the respective $n$-polynomial and $n$-homogeneous model structures for the calculi. 

In \cite{Ta19}, the author introduced the notion of weakly polynomial functors. These functors have a good connectivity relationship with their polynomial approximations. We show, in Section \ref{section: weak poly}, that for weakly polynomial functors, the restricted Taylor tower through realification agrees with the Taylor tower for the pre-realified functor.  The following is Theorem \ref{thm: agreeing tower after realification}. The result does not hold in the complexification induced case, since restriction through complexification only picks out even degree polynomial approximations.

\begin{thm}
Let $F$ be a weakly polynomial reduced orthogonal functor. Then the unitary Taylor tower associated to $r^*F$ is equivalent to the pre-realification of the orthogonal Taylor tower associated to $F$. 
\end{thm}

We leave the homotopy level comparisons here and turn to comparing the model categories in Section \ref{section: model categories}. This section introduces the goal for the remainder of the paper. We give a complete diagram, Figure \ref{fig 1}, of Quillen adjunctions between the model categories for the orthogonal and unitary calculi. The remaining sections of the paper are devoted to demonstrating how Figure \ref{fig 1} commutes. 

We start with the categories of spectra in Section \ref{section: spectra}, and use the change of group functors of Mandell and May \cite{MM02}, to construct Quillen adjunctions between spectra with an action of $\O(n)$, $\U(n)$ and $\O(2n)$ respectively. We utilise the Quillen equivalence between orthogonal and unitary spectra of \cite[Theorem 6.4]{Ta19} to show that these change of group functors interact in a homotopically meaningful way with the change of model functor induced by realification. 

In Section \ref{section: intermediate categories} we move to comparing the intermediate categories. These are categories $\O(n)\E_n^\mathbf{O}$ and $\U(n)\E_n^\mathbf{U}$ constructed by Barnes and Oman \cite{BO13}, and the author \cite{Ta19}, which act as an intermediate in the zig-zag of Quillen equivalences for orthogonal and unitary calculus respectively. For this, we introduce two new intermediate categories, $\O(n)\E_n^\mathbf{U}$ and $\U(n)\E_{2n}^\mathbf{O}$ between the standard intermediate categories. These are the standard intermediate categories with restricted group actions through the subgroup inclusions $\O(n) \hookrightarrow \U(n)$ and $\U(n) \hookrightarrow \O(2n)$. We exhibit Quillen equivalences between these intermediate categories and the standard intermediate categories, completing the picture using change of group functors from \cite{MM02}. The resulting diagram of intermediate categories is as follows,
\[
\xymatrix@C=4em{
\O(n)\E_n^\mathbf{O} \ar@<-1ex>[r]_{r^*}^\sim	&		\O(n)\E_n^\mathbf{U} \ar@<1ex>[r]^{\U(n)_+\wedge_{\O(n)}(-)}\ar@<-1ex>[l]_{r_!}		& 		\U(n)\E_n^\mathbf{U}  \ar@<1ex>[l]^{\iota^*}\ar@<-1ex>[r]_{c^*}											& 		\U(n)\E_{2n}^\mathbf{O} \ar@<1ex>[r]^{\O(2n)_+\wedge_{\U(n)}(-)} \ar@<-1ex>[l]_{c_!}^\sim		&		\O(2n)\E_{2n}^\mathbf{O}.  \ar@<1ex>[l]^{\kappa^*}
}
\]
where $\sim$ denotes a Quillen equivalence. 

Finally, Section \ref{section: homotopy categories} completes the task of showing how Figure \ref{fig 1} commutes by giving commutation results for sub-diagrams of Figure \ref{fig 1} on the homotopy category level.

\subsection*{Notation and Conventions}
The use of a superscript $\mathbf{O}$ is to denote the orthogonal calculus, and a superscript $\mathbf{U}$ is to denote the unitary calculus. When the superscript is omitted, we mean the statement applies to both orthogonal and unitary calculus.

We will refer to the category of based compactly generated weak Hausdorff spaces as the category of based spaces and denote this category by $\T$. This is a cofibrantly generated model category with weak equivalences the weak homotopy equivalences and fibrations the Serre fibrations. The set of generating cofibrations shall be denoted $I$, and the set of generating acyclic cofibrations, denoted $J$.

\subsection*{Acknowledgements}
This work forms part of the authors Ph.D project under the supervision of David Barnes, and has benefited greatly from both helpful and encouraging conversations during this supervision. This work has also benefited from enlightening conversations with Greg Arone. We thank the referee for insightful comments which have greatly improved this article.

%-----------------------------------------------------------------------
%-----------------------------------------------------------------------
\section{The Calculi}\label{section: the calculi}
%-----------------------------------------------------------------------
%-----------------------------------------------------------------------
In this section we give an overview of the theory of orthogonal and unitary calculi. Throughout let $\F$ denote either $\R$ or $\C$ and $\Aut(n) = \Aut(\F^n)$ denote either $\O(n)$ or $\U(n)$. For full details of the theories, see \cite{We95, BO13, Ta19}. 

\subsection{Input Functors} Let $\J$ be the category of finite-dimensional $\F$-inner product subspaces of $\F^\infty$, and $\F$-linear isometries. Denote by $\J_0$ the category with the same objects as $\J$ and morphism space $\J_0(U,V) = \J(U,V)_+$. These categories are $\T$-enriched since $\J(U,V)$ may be topologised as the Stiefel manifold of $\dim_\F (U)$-frames in $V$. These categories are the indexing categories for the functors under consideration in orthogonal and unitary calculus. 

\begin{definition}\label{def: input functors}
Define $\E_0$ to be the category of $\T$-enriched functors from $\J_0$ to $\T$. 
\end{definition}

The category $\E_0^\mathbf{O}$ is category of input functors for orthogonal calculus as studied by Weiss and Barnes and Oman \cite{We95, BO13}. Moreover $\E_0^\mathbf{U}$ is the category of input functors for unitary calculus, studied by the author in \cite{Ta19}. These input categories are categories of diagram spaces as in \cite{MMSS01} hence they can be equipped with a projective model structure. 

\begin{prop}
There is a cellular, proper and topological model category structure on the category $\E_0$, with the weak equivalences and fibrations defined to be the levelwise weak homotopy equivalences and levelwise Serre fibrations respectively. The generating (acyclic) cofibrations are of the form $\J_0(U,-) \wedge i$ where $i$ is a generating (acyclic) cofibration in $\T$. 
\end{prop}

\subsection{Polynomial functors} Arguably the most important class of functors in orthogonal and unitary calculi are the $n$-polynomial functors, and in particular the $n$-th polynomial approximation functor. Here we give a short overview of these functors, for full details on these functors see \cite{We95, BO13, Ta19}.

\begin{definition}
A functor $F \in \E_0$ is \textit{polynomial of degree less than or equal $n$} or equivalently \textit{$n$-polynomial} if the canonical map
\[
F(V) \to \underset{0 \neq U \subseteq \F^{n+1}}{\holim}F(U \oplus V) =: \tau_nF(V)
\]
is a weak homotopy equivalence. 
\end{definition}

\begin{definition}
The \textit{$n$-th polynomial approximation}, $T_nF$, of a functor $F \in \E_0$ is defined to be the homotopy colimit of the sequential diagram
\[
\xymatrix{
F \ar[r]^{\rho} &  \tau_nF \ar[r]^{\rho} &  \tau_n^2F \ar[r]^{\rho} &  \tau_n^3F \ar[r]^{\rho} & \cdots.  
}
\]
\end{definition}

Since an $n$-polynomial functor is $(n+1)$-polynomial, see \cite[Proposition 5.4]{We95}, these polynomial approximation functors assemble into a Taylor tower approximating a given input functor. Moreover there is a model structure on $\E_0$ which captures the homotopy theory of $n$-polynomial functors. 

\begin{prop}[{\cite[Proposition 6.5]{BO13}, \cite[Proposition 2.8]{Ta19}}]
There is a cellular proper topological model structure on $\E_0$ where a map $f\colon E \to F$ is a weak equivalence if $T_nf\colon T_nE \to T_nF$ is a levelwise weak equivalence, the cofibrations are the cofibrations of the projective model structure and the fibrations are levelwise fibrations such that 
\[
\xymatrix{
E \ar[r]^f \ar[d]_{\eta_E} & F \ar[d]^{\eta_F} \\
T_nE \ar[r]_{T_nf} & T_n F
}
\]
is a homotopy pullback square. The fibrant objects of this model structure are precisely the $n$-polynomial functors and $T_n$ is a fibrant replacement functor. We call this the $n$-polynomial model structure and it is denoted $n\poly\E_0$.
\end{prop}

\subsection{Homogeneous functors} The $n$-th layer of the Taylor tower satisfies the property that it is both $n$-polynomial, and its $(n-1)$-st polynomial approximation vanishes, \cite[Example 2.10]{Ta19}. The class of functors which satisfy this property are called $n$-homogeneous. 

\begin{definition}
A functor $F \in \E_0$ is said to be \textit{$n$--reduced} if $T_{n-1}F$ is trivial, and is said to be \textit{homogeneous of degree $n$} or equivalently \textit{$n$-homogeneous} if it is both $n$-polynomial and $n$--reduced.
\end{definition}

There is a further model structure on $\E_0$ which captures the homotopy theory of $n$-homogeneous functors. Denote by $D_nF$ the homotopy fibre of the map
\[
T_nF \longrightarrow T_{n-1}F.
\] 

\begin{prop}[{\cite[Proposition 6.9]{BO13}, \cite[Proposition 3.13]{Ta19}}]
There is a topological model structure on $\E_0$ where the weak equivalences are those maps $f$ such that $D_nf$ is a weak equivalence in $\E_0$, the fibrations are the fibrations of the $n$-polynomial model structure and the cofibrations are those maps with the left lifting property with respect to the acyclic fibrations. The fibrant objects are $n$-polynomial and the cofibrant-fibrant objects are the projectively cofibrant $n$-homogeneous functors.
\end{prop}

In \cite[\S8]{Ta19}, the author gave further characterisations of the $n$-homogeneous model structure. These will prove useful in our comparisons. The results hold true for the orthogonal calculus, with all but identical proofs. 

\begin{prop}[{\cite[Proposition 8.3]{Ta19}}]\label{characterisation of acyclic fibs}
A map $f\colon E \to F$ is an acyclic fibration in the $n$-homogeneous model structure if and only if it is a fibration in the $(n-1)$-polynomial model structure and an $D_n$-equivalence.
\end{prop}

This allows us to characterise the acyclic fibrations between fibrant objects. 

\begin{cor}[{\cite[Corollary 8.4]{Ta19}}]\label{lem: acyclic fibrations in n-homog}
A map $f\colon E \to F$ between $n$-polynomial objects is an acyclic fibration in the $n$-homogeneous model structure if and only if it is a fibration in the $(n-1)$-polynomial model structure.
\end{cor}

We now turn our attention to the cofibrations.

\begin{lem}[{\cite[Lemma 8.5]{Ta19}}]\label{lem: cofibrations of n-homog}
A map $f\colon X \to Y$ is a cofibration in the $n$-homogeneous model structure if and only if it is a projective cofibration and an $(n-1)$-polynomial equivalence. 
\end{lem}

\begin{cor}[{\cite[Corollary 8.6]{Ta19}}]\label{cofibrant objects of n-homog}
The cofibrant objects of the $n$-homogeneous model structure are precisely those $n$-reduced projectively cofibrant objects.
\end{cor}

\subsection{The intermediate categories} In \cite{We95}, Weiss constructs a zig-zag of equivalences between the category of $n$-homogeneous functors (up to homotopy) and the homotopy category of spectra with an action of $\O(n)$. In \cite{BO13}, Barnes and Oman put this zig-zag into a model category theoretic framework via a zig-zag of Quillen equivalences between the $n$-homogeneous model structure on $\E_0^\mathbf{O}$, and spectra with an action of $\O(n)$. This zig-zag moves through an intermediate category, denote $\O(n)\E_n^\mathbf{O}$. In \cite{Ta19}, the author constructs a similar zig-zag of Quillen equivalences between the unitary $n$-homogeneous model structure and spectra with an action of $\U(n)$. We give an overview of the construction of these intermediate categories and how they relate to spectra and the $n$-homogenous model structure. 

Sitting over the space of linear isometries $\J(U,V)$ the the $n$-th complement vector bundle, with total space 
\[
\gamma_n(U,V) = \{ (f,x) \ : \ f \in \J(U,V), x \in \F^n \otimes_\F f(U)^\perp\}
\]
where we have identified the cokernel of $f$ with $f(U)^\perp$, the orthogonal complement of $f(U)$ in $V$. 

\begin{definition}
Define $\J_n$ to be the category with the same objects as $\J$ and morphism space $\J_n(U,V)$ given by the Thom space of the vector bundle $\gamma_n(U,V)$. 
\end{definition}

With this, we may define the intermediate categories. 

\begin{definition}
Define $\E_n$ to be the category of $\T$-enriched functors from $\J_n$ to $\T$, and define the $n$-th intermediate category $\Aut(n)\E_n$ to be the category of $\Aut(n)\T$-enriched functors from $\J_n$ to $\Aut(n)\T$. 
\end{definition}

Let $n\mathbb{S}$ be the functor given by $V \longmapsto S^{nV}$ where $nV := \F^n \otimes_\F V$. By \cite[Proposition 7.4]{BO13} and \cite[Proposition 4.2]{Ta19} the intermediate categories are equivalent to a category of $n\mathbb{S}$-modules and hence come equipped with an $n$-stable model structure similar to the stable model structure on spectra. The weak equivalences of the $n$-stable model structure are given by $n\pi_*$-isomorphisms. Theses are defined via the structure maps of objects in $\Aut(n)\E_n$, and as such have slightly different forms depending on whether one is in the orthogonal or unitary setting. 

For $X \in \O(n)\E_n^\mathbf{O}$,
\[
n\pi_k(X) = \underset{q}{\colim}~ \pi_{k+q}X(\R^q),
\]
and for $Y \in \U(n)\E_n^\mathbf{U}$,
\[
n\pi_k(Y) = \underset{q}{\colim}~ \pi_{k+2q}Y(\C^q).
\]

\begin{prop}[{\cite[Proposition 7.4]{BO13}, \cite[Proposition 5.6]{Ta19}}]\label{prop: n-stable model structure}
There is a cofibrantly generated, proper, topological model structure on the category $\Aut(n)\E_n$, where the weak equivalences are the $n\pi_*$-isomorphisms, the cofibrations are those maps with the left lifting property with respect to all levelwise acyclic fibrations and the fibrations are those levelwise fibrations $f\colon X \to Y$ such that the diagram
\[
\xymatrix{
X(V) \ar[r] \ar[d] & \Omega^{nW}X(V \oplus W) \ar[d] \\
Y(V) \ar[r] & \Omega^{nW}Y(V \oplus W).
}
\]
is a homotopy pullback square for all $V,W \in \J_n$. 
\end{prop}

The fibrant objects of the $n$-stable model structure are called $n\Omega$-spectra and have the property that 
\[
X(V) \longrightarrow \Omega^{nW} X(V \oplus W)
\]
is a levelwise weak equivalence. This property can clearly be deduced from the above diagram by considering the map $X \to \ast$. 

To give the Quillen equivalence between these intermediate categories and spectra with an action of $\Aut(n)$ we now consider the calculi separately. The constructions are similar for both calculi but it is convenient to have different notation for the functors involved. We start with the unitary case.  Define $\alpha_n \colon \J_n^\mathbf{U} \to \J_1^\mathbf{U}$ to be the functor given on objects by $\alpha_n(V) = \C^n \otimes_\C V$, and given on morphisms by $\alpha_n(f,x) = (\C^n \otimes_\C f, x)$. This defines a $\T$--enriched functor, for full details see \cite[Proposition 6.7]{Ta19}.

\begin{prop}[{\cite[Theorem 5.8]{Ta19}}]
There is a series of Quillen equivalences 
\[
\xymatrix@C+1cm{
\U(n)\E_n^\mathbf{U} \ar@<1ex>[r]^{(\alpha_n)_!} & \s^\mathbf{U}[\U(n)] \ar@<1ex>[l]^{(\alpha_n)^*} \ar@<1ex>[r]^{r_!} & \s^\mathbf{O}[\U(n)] \ar@<1ex>[l]^{r^*}\\
}
\]
with $(\alpha_n)^* \Theta (V) = \Theta(\C^n \otimes_\C V)$, and $(\alpha_n)_!$ is the left Kan extension along $\alpha_n$. 
\end{prop}

The orthogonal case is similar, full details may be found in \cite[\S 8]{BO13}.

\begin{prop}[{\cite[Proposition 8.3]{BO13}}]
There is a Quillen equivalence
\[
\adjunction{(\beta_n)_!}{\O(n)\E_n^\mathbf{O}}{\s^\mathbf{O}[\O(n)]}{(\beta_n)^*}
\]
with $(\beta_n)^* \Theta (V) = \Theta(\R^n \otimes_\R V)$, and $(\beta_n)_!$ is the left Kan extension along $\beta_n$. 
\end{prop}

\subsection{The derivatives of a functor} We now move on to discussing the derivatives of a functor. The derivatives are naturally objects in $\Aut(n)\E_n$. Their definition comes from constructing an adjunction between $\E_0$ and $\Aut(n)\E_n$. The inclusion $\F^m \to \F^n$ onto the first $m$-coordinates induces a functor $i_m^n \colon \J_m \to \J_n$. 

\begin{definition}
Define the \textit{restriction functor} $\res_0^n \colon \E_n \to \E_m$ to be precomposition with $i_m^n$, and define the \textit{induction functor} $\ind_m^n \colon \E_m \to \E_n$ to be the right Kan extension along $i_m^n$. In the case $m=0$, the induction functor $\ind_0^n$ is called the \textit{$n$-th derivative}. 
\end{definition}

Combining this adjunction with a change of group action from \cite{MM02} provides an adjunction
\[
\adjunction{\res_0^n/\Aut(n)}{\Aut(n)\E_n}{\E_0}{\ind_0^n \varepsilon^*}.
\]

This adjunction is a Quillen equivalence between the $n$-homogeneous model structure on $\E_0$ and the $n$-stable model structure on $\Aut(n)\E_n$. We will refer to the right adjoint as inflation-induction. 

\begin{prop}[{\cite[Theorem 10.1]{BO13}, \cite[Theorem 6.5]{Ta19}}]
The adjoint pair
\[
\adjunction{\res_0^n/\Aut(n)}{\Aut(n)\E_n}{n\homog\E_0}{\ind_0^n\varepsilon^*}
\]
is a Quillen equivalence.
\end{prop}

\subsection{Classification of $n$-homogeneous functors}
For a functor $F \in \E_0^\mathbf{U}$, inflation-induction and the left adjoint to $(\alpha_n \circ r)^*$ determine a spectrum $\Psi_F^n$ with an action of $\U(n)$. That is, $\Psi_F^n = (\alpha_n \circ r)_!\ind_0^n\varepsilon^*F$. Moreover, for $F \in \E_0^\mathbf{O}$, inflation-induction and the left adjoint to $(\beta_n)^*$ defines a spectrum with an action of $\O(n)$, which we again denote by $\Psi_F^n$.

\begin{prop}[{\cite[Theorem 7.3]{We95},\cite[Theorem 7.1]{Ta19}}]\label{homog characterisation}
Let $F \in \E_0$ be $n$-homogeneous for some $n >0$. Then $F$ is levelwise weakly equivalent to the functor defined as 
\begin{equation*}\label{char of homog functors}
U \longmapsto \Omega^\infty [(S^{nU} \wedge \Psi_F^n)_{h\Aut(n)}].
\end{equation*}
\end{prop}

\subsection{Weak Polynomials} An important class of functors, introduced by the author in \cite{Ta19} are the weak polynomial functors. These functors have a good connectivity relationship with the polynomial approximations and result in a convergent Taylor tower. We give an overview of the theory here, noting that the proofs provided by the author in \cite[\S9]{Ta19} work in the orthogonal setting also.

\begin{definition}\label{unitary agreement}
A map $p\colon F \to G$ in $\E_0$ is \textit{an order $n$ agreement} if there is some $\rho \in \mathbb{N}$ and $b \in \Z$ such that $p_U\colon F(U) \to G(U)$ is $((n+1)\dim_\R(U)-b)$-connected for all $U \in \J_0$, satisfying $\dim_\F(U) \geq \rho$. We will say that \textit{$F$ agrees with $G$ to order $n$} if there is an order $n$ unitary agreement $p\colon F \to G$ between them.
\end{definition}

When two functors agree to a given order, their Taylor tower agree to a prescribed level. The first result in that direction is the unitary analogue of \cite[Lemma e.3]{We98}.

\begin{lem}[{\cite[Lemma e.3]{We98},\cite[Lemma 9.5]{Ta19}}]\label{connected argument}
let $p \colon G \to F$ be a map in $\E_0$. Suppose that there is $b \in \Z$ such that $p_U\colon G(U) \to F(U)$ is $((n+1)\dim_\R(U) - b)$-connected for all $U \in \J_0$ with $\dim_\F(U) \geq \rho$. Then
\[
\tau_n(p)_U \colon \tau_n (G(U)) \to \tau_n(F(U))
\] 
is $((n+1)\dim_\R(U) -b +1)$-connected for all $U \in \J_0$. 
\end{lem}

Iterating this result, gives the following. 

\begin{lem}[{\cite[Lemma e.7]{We98},\cite[Lemma 9.6]{Ta19}}]\label{agreement gives agreeing polynomials}
If $p \colon F \to G$ is an order $n$ agreement, then $T_k F \to T_k G$ is a levelwise weak equivalence  for $k \leq n$.
\end{lem}

Agreement with the $n$-polynomial approximation functor for all $n\geq 0$ gives convergence of the Taylor tower. 

\begin{lem}[{\cite[Lemma 9.10]{Ta19}}]
If for all $n\geq 0$, a unitary functor $F$ agrees with $T_nF$ to order $n$ then the Taylor tower associated to $F$ converges to $F(U)$ at $U$ with $\dim_\F(U) \geq \rho$.
\end{lem}

\begin{definition}\label{def: weak poly}
A unitary functor $F$ is \textit{weakly $(\rho,n)$-polynomial} if the map $\eta \colon F(U) \to T_nF(U)$ is an agreement of order $n$ whenever $\dim_\F(U) \geq \rho$. A functor is \textit{weakly polynomial} if it is weakly $(\rho,n)$-polynomial for all $n\geq 0$. 
\end{definition}

\begin{rem}
In the above definition of weakly polynomial, we require that the functor is weakly $(\rho, n)$-polynomial for all $n$. Here $\rho$ is permitted to depend on $n$, i.e. the functor may be weakly $(\rho_n, n)$-polynomial for all $n$, so long as, the sequence $(\rho_n)_{n \geq 0}$ is bounded above, in which case, one may take $\rho$ to the the upper bound of this sequence, hence why we have fixed a $\rho$ in the definition.
\end{rem}

%The following result is useful for identifying weakly polynomial functors.

%\begin{thm}[{\cite[Theorem 9.14]{Ta19}}]\label{quasi thm} 
%Let $E, F\in \E_0$ are such that there is a homotopy fibre sequence 
%\[
%E(U) \to F(U) \to  F(U \oplus V)
%\]
%for $U,V \in \J$. Then 
%\begin{enumerate}
%\item If $F$ is weakly $(\rho,n)$-polynomial, then $E$ is weakly $(\rho,n)$-polynomial; and
%\item If $E$ is weakly $(\rho,n)$-polynomial and $F(U)$ is $1$-connected whenever $\dim_\F(U) \geq \rho$, then $F$ is weakly $(\rho, n)$-polynomial.
%\end{enumerate}
%\end{thm}

%-----------------------------------------------------------------------
%-----------------------------------------------------------------------
\section{Comparing the input functors}\label{section: input functors}
%-----------------------------------------------------------------------
%-----------------------------------------------------------------------
Let $V \in \J_0^\mathbf{O}$, then the complexification of $V$, $\C \otimes_\R V$, is a complex vector space such that 
\[
\dim_\C \C \otimes_\R V = \dim_\R V.
\]
Given a $\R$-inner product $\langle-,-\rangle_V$ on $V$, there is a well defined $\C$-inner product on $\C \otimes V$, given by 
\[
\langle (a+ib) \otimes v, (c+id) \otimes w \rangle = \langle av, cw \rangle_V + \langle bv, dw \rangle_V + i\langle bv, cw \rangle_V -i\langle av, dw\rangle_V
\]
where $ a + i b, c+id\in \C$, and $v, w \in V.$

The complexification of an $\R$-linear map $T$ is given by $T_\C = \C \otimes T$. Moreover in the finite dimensional case the matrices representing $T$ and $\C \otimes T$ are equal (corresponding to the inclusion $\O(n) \hookrightarrow \U(n)$) and we get characterisations of images and kernels,
\[
\ker(\C \otimes_\R T) = \C \otimes_\R \ker(T) \ \text{and} \  \mathrm{im}(\C \otimes_\R T) = \C \otimes_\R \mathrm{im}(T).
\]

Given a $\R$-linear isometry, $T\colon V \to W$, $\C \otimes T\colon \C \otimes V \to \C \otimes W, c \otimes v \longmapsto c \otimes T(v)$, is a $\C$-linear isometry, that respects the inner product. It follows that complexification gives a well defined functor $c\colon \J_0^\mathbf{O} \to \J_0^\mathbf{U}$. 

The ``opposite operation'' to complexification is that of realification. Let $W$ be a complex vector space, then its realification $W_\R$  is the set $W$ with vector addition and scalar multiplication by reals inherited unchanged from $W$ and the complex multiplication ``forgotten''. If $\{e_1, \dots, e_n\}$ is a basis for $W$ then $\{e_1, \dots, e_n, ie_1, \dots , ie_n\}$ is a basis for $W_\R$. It follows that
\[
\dim_\R W_\R = 2\dim_\C W = \dim_\R W.
\]

Up to isomorphism it suffices to check that there is a well defined inner product on the realification of $\C^n$ induced by the Hermitian inner product on $\C^n$. Recall for vectors $\mathbf{c}=(c_i), \mathbf{c'}=(c_i')$ in $\C^n$, the Hermitian inner product is given by 
\[
\langle \mathbf{c}, \mathbf{c'} \rangle_\C = \sum_{i=1}^n c_i\overline{c_i'}.
\]
To obtain a real inner product on $\R^{2n} = (\C^n)_\R$, we realise the vectors $\mathbf{c}$ and $\mathbf{c'}$ as $\mathbf{c}= \mathbf{a} +i\mathbf{b}$ and $\mathbf{c'} = \mathbf{a'} + i\mathbf{b}$, where $\mathbf{a},\mathbf{a'},\mathbf{b}, \mathbf{b'} \in \R^n$. By taking the real part of $\langle \mathbf{c}, \mathbf{c}'\rangle_\C$, we then define a real inner product on $\R^{2n}$ as 
\[
\langle (\mathbf{a},\mathbf{b}), (\mathbf{a'} , \mathbf{b'})\rangle_\R= \Re \left( \sum_{i=1}^n c_i\overline{c_i'}\right) = \Re \left( \langle \mathbf{c}, \mathbf{c'} \rangle \right)
\]
under the identification $c_j = a_j+ ib_j$, $c_j' = a_j' +ib_j'$ and where $(\mathbf{a}, \mathbf{b})$ is notation for the vector
\[
(a_1, b_1, a_2, b_2, \cdots, a_n, b_n) \in \R^{2n}.
\]

If $T\colon \C^k \to \C^m$ is a $\C$-linear map then we may view it as a $\R$-linear map $T_\R \colon (\C^k)_\R \to (\C^m)_\R$. 
%In particular, thinking of $T$ as a matrix, we can write $T = (t_{ij})_{i,j}$, we can rewrite $T$ as 
%\[
%\begin{pmatrix}
%a_{11} + ib_{11} & a_{12}+ib_{12} & \cdots & a_{1k} +ib_{1k} \\
%\vdots & \vdots & & \vdots \\
%a_{m1} + ib_{m1} & a_{m2}+ib_{m2} & \cdots & a_{mk} +ib_{mk} \\
%\end{pmatrix}
%\]
%by rewriting each entry as $t_{ij} = a_{ij} + ib_{ij}$ for $a_{ij}, b_{ij} \in \R$. Then $T_\R$ has matrix representation of the form
%\[
%\begin{pmatrix}
%a_{11}  & -b_{11} & a_{12} &  -b_{12} & \cdots & a_{1k}  & -b_{1k} \\
%b_{11}  & a_{11} & b_{12}  & a_{12} & \cdots & b_{1k}  & a_{1k} \\
%\vdots & \vdots & \vdots & \vdots & & \vdots & \vdots \\
%a_{m1} & -b_{m1} & a_{m2} & -b_{m2} & \cdots & a_{mk} & -b_{mk} \\
%b_{m1} & a_{m1} & b_{m2} & a_{m2} & \cdots & b_{mk} & a_{mk} \\
%\end{pmatrix}
%\]
%where we block decompose each individual entry. 
It follows that 
\[
\ker(T_\R) = (\ker(T))_\R \ \text{and} \ \mathrm{im}(T_\R) = (\mathrm{im}(T))_\R.
\]

If $T\colon V \to W$ is a $\C$-linear isometry, then $T_\R \colon V_\R \to W_\R$ is a $\R$-linear isometry, and it follows that realification gives a well defined functor $r\colon \J_0^\mathbf{U} \to \J_0^\mathbf{O}$.
 
\subsection{Realification and complexification induce Quillen functors}

For an orthogonal functor $F \in \E_0^\mathbf{O}$, precomposition with $r$, which we call ``pre-realification'' defines a unitary functor 
\[
r^* F \colon \J_0^\mathbf{U} \to \T.
\]

Hence pre-realification defines a functor $r^* \colon \E_0^\mathbf{O} \to \E_0^\mathbf{U}$, which has a left adjoint $r_!$ given by the formula,
\[
(r_! E)(V) = \int^{W \in \J_0^\mathbf{U}} E(W) \wedge \J_0^\mathbf{O}(W_\R, V),
\]
i.e. $r_!$ is the left Kan extension along $r$. 

Similarly, complexification defines a functor $c^* \colon \E_0^\mathbf{U} \to \E_0^\mathbf{O}$, which has left adjoint $c_!$ given by the left Kan extension along $c$. These functors are homotopically meaningful when one considers the projective model structures on the categories of input functors.

\begin{lem}\label{QA base level for r}
The adjoint pair
\[
\adjunction{r_!}{\E_0^\mathbf{U}}{\E_0^\mathbf{O}}{r^*}
\]
is a Quillen adjunction, when both categories are equipped with their projective model structures. 
\end{lem}
\begin{proof}
Let $f\colon E \to F$ be a levelwise fibration (resp. levelwise weak equivalence). Then by definition $r^*f \colon r^*E \to r^* F$ is a levelwise fibration (resp. levelwise weak equivalence). Hence $r^*$ preserves fibrations and acyclic fibrations. 
\end{proof}

\begin{lem}\label{QA base level for c}
The adjoint pair
\[
\adjunction{c_!}{\E_0^\mathbf{O}}{\E_0^\mathbf{U}}{c^*}
\]
is a Quillen adjunction, when both categories are equipped with their projective model structures. 
\end{lem}
\begin{proof}
The proof is all but identical to that of Lemma \ref{QA base level for r}. 
\end{proof}

%-----------------------------------------------------------------------
%-----------------------------------------------------------------------
\section{Comparing the polynomial and homogeneous functors}\label{section: poly and homog}
%-----------------------------------------------------------------------
%-----------------------------------------------------------------------

\subsection{Homogeneous functors} 

Heuristically, the homogeneous functors are the building blocks of the Taylor towers of the calculi. As such we start with a direct comparison between these functors. This comparison is reliant on the classifications of homogeneous functors from orthogonal and unitary calculi. 

\begin{lem}\label{homog r}
If an orthogonal functor $F$ is $n$-homogeneous, then $r^*F$ is $n$-homogeneous. 
\end{lem}
\begin{proof}
Let $F$ be an $n$-homogeneous orthogonal functor. Then by the characterisation, Proposition \ref{homog characterisation}, $F$ is levelwise weakly equivalent to the functor 
\[
V \longmapsto \Omega^\infty[(S^{\R^n \otimes_\R V} \wedge \Psi_F^n)_{h\O(n)}]
\]
where $\Psi_F^n$ is an orthogonal spectrum with an $O(n)$-action. It follows that pre-realification of $F$ is levelwise weakly equivalent to the functor 
\[
W \longmapsto \Omega^\infty[(S^{\R^n \otimes_\R W_\R} \wedge \Psi^n_F)_{h\O(n)}].
\]
Using the derived change of group functor, we construct an orthogonal spectrum with an action of $\U(n)$, 
\[
\U(n)_+ \wedge_{\O(n)}^\mathds{L} \Psi_F^n := \U(n)_+ \wedge_{\O(n)} (E\O(n)_+ \wedge \Psi_F^n).
\]
By the classification of $n$-homogeneous unitary functors, Proposition \ref{homog characterisation}, there is an $n$-homogeneous functor $F'$ associated to the above spectrum, given by 
\[
W \longmapsto \Omega^\infty[(S^{\C^n \otimes_\C W} \wedge (\U(n)_+ \wedge_{\O(n)} (E\O(n)_+ \wedge \Psi_F^n)))_{h\U(n)}].
\]
By \cite[Proposition V.2.3]{MM02}, $F'(W)$ is isomorphic to 
\[
\Omega^\infty[\U(n)_+ \wedge_{\O(n)} ((\iota^*S^{\C^n \otimes_\C W} \wedge(E\O(n)_+ \wedge \Psi_F^n)))_{h(\U(n)}].
\]
The $\U(n)$-action on $\U(n)_+ \wedge_{\O(n)} ((\iota^*S^{\C^n \otimes_\C W} \wedge(E\O(n)_+ \wedge \Psi_F^n)))$ is free ($(E\O(n)_+$ is a free $\O(n)$-space), hence taking homotopy orbits equates to taking strict orbits. Hence there is an isomorphism
\[
F'(W) \cong \Omega^\infty[(\U(n)_+ \wedge_{\O(n)} ((\iota^*S^{\C^n \otimes_\C W} \wedge(E\O(n)_+ \wedge \Psi_F^n))))/\U(n)].
\]
The strict $\U(n)$-orbits of the spectrum $\U(n)_+ \wedge_{\O(n)} ((\iota^*S^{\C^n \otimes_\C W} \wedge(E\O(n)_+ \wedge \Psi_F^n))$ are isomorphic to the $\O(n)$-orbits of the spectrum, $\iota^*S^{\C^n \otimes_\C W} \wedge(E\O(n)_+ \wedge \Psi_F^n))$, hence $F'(W)$ is isomorphic to 
\[
\Omega^\infty[(\iota^*S^{\C^n \otimes_\C W} \wedge(E\O(n)_+ \wedge \Psi_F^n)))/\O(n)].
\]
This last is precisely
\[
\Omega^\infty[(\iota^*S^{\C^n \otimes_\C W} \wedge \Psi_F^n)_{h\O(n)}] 
\]
as homotopy orbits is the left derived functor of strict orbits and smashing with $E\O(n)_+$ is a cofibrant replacement in the projective model structure.

Since the action of $\O(n)$ on $\iota^*S^{\C^n \otimes_\C W}$ is equivalent to the $\O(n)$ action on $S^{\R^n \otimes_\R W_\R}$ and the one-point compactification are isomorphic, the above infinite loop space is isomorphic to
\[
\Omega^\infty[(S^{\R^n \otimes_\R W_\R} \wedge \Psi_F^n)_{h\O(n)}].
\]
By the characterisation of $n$-homogeneous orthogonal functors, we see that this is levelwise weakly equivalent to $F(W_\R) = (r^*F)(W).$
\end{proof}

\begin{lem}\label{homog for c}
If a unitary functor $E$ is $n$-homogeneous, then $c^*E$ is $(2n)$-homogeneous. 
\end{lem}
\begin{proof}
Since $E$ is $n$-homogeneous,
\[
E(W) \simeq \Omega^\infty [(S^{\C^n \otimes W} \wedge \Psi_E^n)_{h\U(n)}].
\]
By definition 
\[
(c^*E)(V) = E(\C \otimes_\R V) \simeq  \Omega^\infty [(S^{\C^n \otimes_\C \C \otimes_\R V} \wedge \Psi_E^n)_{hU(n})].
\]
Observe that 
\[
\O(2n)_+ \wedge^\mathds{L}_{\U(n)} \Psi_E^n
\]
is an orthogonal spectrum with $\O(2n)$-action. The classification of homogeneous functors in orthogonal calculus, Proposition \ref{homog characterisation} gives a $(2n)$-homogeneous functor, 
\[
V \longmapsto \Omega^\infty [(S^{\R^{2n} \otimes V} \wedge (\O(2n)\wedge^\mathds{L}_{\U(n)} \Psi_E^n) )_{h\O(2n)}].
\]
A similar argument to Lemma \ref{homog r} yields the result. 
\end{proof}

\subsection{Polynomial functors} 

Using the above results on pre-realification and pre-complexification of homogeneous functors, we can compare polynomial functors. We must add the technical assumption that $F$ is reduced, i.e., that $F(\R^\infty)$ is weakly contractible. Many of the functors which one wishes to consider in the calculi are reduced, and in the situations where they are not, we can take their reduced part to be the homotopy fibre of the map $F \to T_0F$, and work relative to the $0$--polynomial approximation. 

\begin{thm}\label{poly for r}
If an orthogonal functor $F$ is $n$-polynomial and $F(\R^\infty)$ is weakly contractible, then $r^*F$ is an $n$-polynomial unitary functor, that is, the map
\[
r^*T_n^\mathbf{O} F \longrightarrow T_n^\mathbf{U}(r^*T_n^\mathbf{O} F)
\]
is a levelwise weak equivalence for every $F \in \E_0^\mathbf{O}$. 
\end{thm}
\begin{proof}
We argue by induction on the polynomial degree. The case $n=0$ follows by definition. Assume the map $r^*T_{n-1}^\mathbf{O} F \to T_{n-1}^\mathbf{U}(r^*T_{n-1}^\mathbf{O} F)$ is a levelwise weak equivalence. There is a homotopy fibre sequence 
\[
T_n^\mathbf{O} F \longrightarrow T_{n-1}^\mathbf{O} F \longrightarrow R_n^\mathbf{O} F
\]
where $R_n^\mathbf{O} F$ is $n$-homogeneous, since $F$ satisfies the conditions of \cite[Corollary 8.3]{We95}. Lemma \ref{homog r} implies that $r^* R_n^\mathbf{O} F$ is $n$-homogeneous in $\E_0^\mathbf{U}$, and in particular $n$-polynomial. As homotopy fibres of maps between $n$-polynomial objects are $n$-polynomial, the homotopy fibre of the map $r^*T_{n-1}^\mathbf{O} F \to r^* R_n^\mathbf{O} F$ is $n$-polynomial. Computation of homotopy fibres is levelwise, hence the homotopy fibre in question is $r^* T_n^\mathbf{O} F$, and it follows that 
\[
r^* T_n^\mathbf{O} F \longrightarrow T_n^\mathbf{U}(r^* T_n^\mathbf{O} F)
\]
is a levelwise weak equivalence. 
\end{proof}

\begin{thm}\label{n-poly for c}
If an unitary functor $E$ is $n$-polynomial and $E(\C^\infty)$ is weakly contractible,  then $c^*F$ is $(2n)$-polynomial, that is, the map
\[
c^*T_n^\mathbf{U} E \longrightarrow T_{2n}^\mathbf{O}(c^*T_n^\mathbf{U} E)
\]
is a levelwise weak equivalence for all $E \in \E_0^\mathbf{U}$.
\end{thm}
\begin{proof}
The argument follows as in Theorem \ref{poly for r} using Lemma \ref{homog for c} in place of Lemma \ref{homog r}.
\end{proof}

\subsection{Polynomial model structures}
We turn our attention to a model structure comparison. Since Theorem \ref{poly for r} only applies to reduced functors (those $F$ such that $F(\R^\infty)$ is trivial) we have to restrict our attention to reduced functors. We do this by defining the reduced part of $F$, denoted $\mathrm{red}(F)$ to be the homotopy fibre of the map $F \to T_0F$. Then $\mathrm{red}(F)$ is reduced, and defines a functor $\mathrm{red}(-): \E_0 \to \E_0$.  The composite
\[
\E_0^\mathbf{O} \xrightarrow{\mathrm{red}(-)} \E_0^\mathbf{O} \xrightarrow{\quad r^* \quad} \E_0^\mathbf{U},
\]
is a right Quillen functor, hence induces a functor on the homotopy categories, although it does not have a left adjoint. In Section \ref{section: homotopy categories} we restrict our attention to the homotopy categories, and do not require the existence of a left adjoint to $r^*\mathrm{red}(-)$. 

\begin{lem}\label{QA for n-poly for r}
The composite
\[
n\poly\E_0^\mathbf{O} \xrightarrow{\mathrm{red}(-)} n\poly\E_0^\mathbf{O} \xrightarrow{\quad r^* \quad} n\poly\E_0^\mathbf{U}
\]
preserves acyclic fibrations, and fibrations. In particular, there is an induced functor on homotopy categories.
\end{lem}
\begin{proof}
The acyclic fibrations in the $n$--polynomial model structure are levelwise, hence both terms in the composite, and resultingly their composite, preserves these. It suffices by \cite[Corollary A.2]{Du01} to show that the composite preserves fibrations between fibrant objects, which are the levelwise fibrations by \cite[Proposition 3.3.16]{Hi03}. It hence suffices to show that $F$ preserves fibrant objects. Let $F$ be a $n$--polynomial orthogonal functor, then $\mathrm{red}(F)$ is also $n$--polynomial, and reduced. An application of Theorem \ref{poly for r} implies the result. 
\end{proof}

\begin{lem}\label{QA base level for c}
The composite
\[
n\poly\E_0^\mathbf{U} \xrightarrow{\mathrm{red}(-)} n\poly\E_0^\mathbf{U} \xrightarrow{\quad c^* \quad} (2n)\poly\E_0^\mathbf{O}
\]
preserves acyclic fibrations and fibrations. In particular, there is an induced functor on homotopy categories.
\end{lem}
\begin{proof}
This follows similarly to Lemma \ref{QA for n-poly for r}, using Theorem \ref{n-poly for c} in place of Theorem \ref{poly for r}.
\end{proof}

\subsection{Homogeneous model structures} The homogeneous model structures are right Bousfield localisations of the $n$-polynomial model structures. This fact, together with the further characterisations of the homogeneous model structure provided by the author in \cite{Ta19}, allow for the construction of Quillen functors between the orthogonal $n$-homogeneous model structure and the unitary $n$-homogeneous model structure. 

\begin{prop}\label{QA homog r}
The composite
\[
n\homog\E_0^\mathbf{O} \xrightarrow{\mathrm{red}(-)} n\homog\E_0^\mathbf{O} \xrightarrow{\quad r^* \quad} n\homog\E_0^\mathbf{U}
\]
preserves acyclic fibrations and fibrations. In particular, there is an induced functor on homotopy categories.
\end{prop}
\begin{proof}
First suppose that $f\colon E \to F$ is a fibration in $n\homog\E_0^\mathbf{O}$. It follows that $f$ is a fibration in the $n$-polynomial model structure, and hence $r^*\mathrm{red}(f)$ is a fibration in $n\poly\E_0^\mathbf{U}$ by Lemma \ref{QA for n-poly for r}, and hence $n\homog\E_0^\mathbf{U}$.

Suppose further that $f$ in an acyclic fibration in $n\homog\E_0^\mathbf{O}$. By Proposition \ref{characterisation of acyclic fibs}, it follows that $f$ is an $(n-1)$-polynomial fibration and an $D_n^\mathbf{O}$-equivalence. As above it follows that $r^*\mathrm{red}(f)$ is a fibration in $(n-1)\poly\E_0^\mathbf{U}$. In particular the homotopy fibre of $r^*\mathrm{red}(f)$ is $(n-1)$-polynomial. Such objects are trivial in $n\homog\E_0^\mathbf{U}$ hence, since $n\homog\E_0^\mathbf{U}$ is stable, $r^*\mathrm{red}(f)$ is a weak equivalence in $n\homog\E_0^\mathbf{U}$. 
\end{proof}

\begin{prop}\label{QA base level for c}
The composite
\[
n\homog\E_0^\mathbf{U} \xrightarrow{\mathrm{red}(-)} n\homog\E_0^\mathbf{U} \xrightarrow{\quad c^* \quad} (2n)\homog\E_0^\mathbf{O}
\]
preserves acyclic fibrations and fibrations. In particular, there is an induced functor on homotopy categories.
\end{prop}
\begin{proof}
The proof follows almost verbatim from Proposition \ref{QA homog r}.
\end{proof}

\begin{rem}
Without a clearer understanding on how the pre-realification and pre-complexification functors behave with respect to the polynomial approximations, it is not possible to say that they preserve all $n$-homogeneous equivalences. In particular, if $f\colon X \to Y$ is an $n$-homogeneous equivalence then $D_n X$ is levelwise weakly equivalent to $D_n Y$, and there is a diagram of homotopy fibre sequences the form
\[
\xymatrix{
D_n X \ar[r] \ar[d]^\simeq & T_n X \ar[r] \ar[d] & T_{n-1} X \ar[d] \\
D_n Y \ar[r] & T_n Y \ar[r] & T_{n-1} Y,
}
\]
which after applying $r^*$ or $c^*$ results in diagrams of homotopy fibre sequences (since fibre sequences are defined levelwise)
\[
\xymatrix{
r^*D_n^\mathbf{O} X \ar[r] \ar[d]^\simeq & r^*T_n^\mathbf{O} X \ar[r] \ar[d] & r^*T_{n-1}^\mathbf{O} X \ar[d] \\
r^*D_n^\mathbf{O} Y \ar[r] & r^*T_n^\mathbf{O} Y \ar[r] & r^*T_{n-1}^\mathbf{O} Y.
}
\hspace{1cm}
\xymatrix{
c^*D_n^\mathbf{U} X \ar[r] \ar[d]^\simeq & c^*T_n^\mathbf{U} X \ar[r] \ar[d] & c^*T_{n-1}^\mathbf{U} X \ar[d] \\
c^*D_n^\mathbf{U} Y \ar[r] & c^*T_n^\mathbf{U} Y \ar[r] & c^*T_{n-1}^\mathbf{U} Y.
}
\]
Since we do not have a useful relation between $T_n^\mathbf{U}(r^*X)$ and $r^*T_n^\mathbf{O}(X)$, nor between $T_n^\mathbf{O}(c^*X)$ and $c^*T_n^\mathbf{U}(X)$, it is difficult to saying anything meaningful about how the above diagrams relate to the following diagram
\[
\xymatrix{
D_n^\mathbf{U}(r^*X) \ar[r] \ar[d] & T_n^\mathbf{U}(r^*X) \ar[r] \ar[d] & T_{n-1}^\mathbf{U}(r^*X) \ar[d]\\
D_n^\mathbf{U}(r^*Y) \ar[r]& T_n^\mathbf{U}(r^*Y) \ar[r] & T_{n-1}^\mathbf{U}(r^*Y).
}
\hspace{1pt}
\xymatrix{
D_n^\mathbf{O}(c^*X) \ar[r] \ar[d] & T_n^\mathbf{O}(c^*X) \ar[r] \ar[d] & T_{n-1}^\mathbf{O}(c^*X) \ar[d]\\
D_n^\mathbf{O}(c^*Y) \ar[r]& T_n^\mathbf{O}(c^*Y) \ar[r] & T_{n-1}^\mathbf{O}(c^*Y).
}
\]
\end{rem}

%-----------------------------------------------------------------------
%-----------------------------------------------------------------------
\section{Comparing weakly polynomial functors}\label{section: weak poly}
%-----------------------------------------------------------------------
%-----------------------------------------------------------------------

\subsection{Agreement}
The notion of agreement plays a central role in the theory or orthogonal and unitary calculus, for example it is crucial to the proof that the $n$-th polynomial approximation in $n$-polynomial, see \cite{We98}. The pre-realification and pre-complexification functors behave well with respect to functors which agree to a certain order.

\begin{lem}\label{realification of an agreement}
If a map $p\colon F \to G$ in $\E_0^\mathbf{O}$ is an order $n$ orthogonal agreement, then $r^*p \colon r^*F \to r^*G$ in $\E_0^\mathbf{U}$ is an order $n$ unitary agreement. 
\end{lem}
\begin{proof}
Since $p$ is an order $n$ orthogonal agreement, there is an integer $b \in \Z$, such that $p_V\colon F(V) \to G(V)$ is $(-b + (n+1) \dim_\R V)$-connected. It follows by definition that  $(r^*p)_W = p_{W_\R} \colon F(W_\R) \to G(W_\R)$ is $(-b + (n+1) \dim_\R W_\R)$-connected. Since $\dim_\R W_\R = \dim_\R W$, it follows that $(r^*p)_W$ is $(-b +(n+1) \dim_\R W)$-connected, and hence $r^*p$ is an order $n$ unitary agreement. 
\end{proof}

\begin{lem}\label{complexification of an agreement}
If a map $p\colon F \to G$ in $\E_0^\mathbf{U}$ is an order $n$ unitary agreement, then $c^*p\colon c^*F \to c^*G$ in $\E_0^\mathbf{O}$ is an order $2n$ unitary agreement. 
\end{lem}
\begin{proof}
Since $p$ is an order $n$ unitary agreement, there is an integer $b$, such that $p_V\colon F(V) \to G(V)$ is $(-b + (n+1) \dim_\R V)$-connected. It follows by definition that  $(c^*p)_W = p_{\C \otimes W} \colon F(\C \otimes W) \to G(\C \otimes W)$ is $(-b + (n+1) \dim_\R (\C \otimes W))$-connected. Since $\dim_\R \C \otimes W = 2\dim_\C \C\otimes W =  2\dim_\R W$, it follows that $(c^*p)_W$ is $(-b +2(n+1) \dim_\R W)$-connected, and hence in particular $(-b + (2n+1)\dim_\R W)$-connected, hence $c^*p$ is an order $2n$ unitary agreement. 
\end{proof}

\begin{rem}\label{unable to compare towers}
In \cite{BE16} Barnes and Eldred give a tower level comparison between Goodwillie calculus and orthogonal calculus. This relies on the functor $F$ from Goodwillie calculus being stably $n$-excisive, that is, the functor must behave well with respect to (co)Cartesian cubes, see \cite[Definition 4.1]{Go91} for the precise definition. The key property gained by a stably $n$-excisive functor is that the polynomial approximation map $p_n \colon F \to P_nF$ in Goodwillie calculus is an agreement of order $n$ in the Goodwillie calculus setting. This allows for a clear comparison between the $n$-polynomial approximation functors of Goodwillie and orthogonal calculi. In general, the map $F \to T_nF$ is not an agreement of order $n$ in either orthogonal or unitary calculus. 
\end{rem}

%\begin{ex}
%For each $k \geq 0$, define a functor $F_k \in \E_0$ by
%\[
%F_k(V) = 
%\begin{cases}
%\ast \ \text{if} \ \dim(V)\neq k \\
%S^0 \ \text{if} \ \dim(V)=k.
%\end{cases}
%\]
%Let $F \in \E_0$ be defined by $F = \prod_{k=0}^\infty F_k$. Then $T_n F_k$ is trivial, hence so too is $T_nF$. 
%\end{ex}

\subsection{Weak Polynomial} Despite the fact that not all functors agree to a specific order with their polynomial approximations, there is a large class of functors which do. These functors, which were introduced in \cite{Ta19}, are called weakly polynomial and interact meaningfully with the comparisons. 

\begin{prop}\label{prop: weak poly give agreeing approxs} \hspace{10cm}
\begin{enumerate}
\item If $F \in \E_0^\mathbf{O}$ is reduced and $(\rho,n)$-polynomial then the map
\[
T_n^\mathbf{U}(r^*\eta_n)\colon T_n^\mathbf{U}(r^*F) \to T_n^\mathbf{U}(r^*T_n^\mathbf{O}F)
\]
is a levelwise weak equivalence. Thus, the $n$-th polynomial approximation of $r^*F$ is given by the map $r^*F \to r^*(T_n^\mathbf{O}F)$. 
\item If $F \in \E_0^\mathbf{U}$ is reduced and $(\rho,n)$-polynomial then 
\[
T_{2n}^\mathbf{O}(c^*\eta_n)\colon T_{2n}^\mathbf{O}(c^*F) \to T_{2n}^\mathbf{O}(c^*T_n^\mathbf{U}F)
\]
is a levelwise weak equivalence. Thus, the $(2n)$-th polynomial approximation of $c^*F$ is given by the map $c^*F \to c^*(T_n^\mathbf{O}F)$. 
\end{enumerate}
\end{prop}
\begin{proof}
We prove part $(1)$, part $(2)$ is similar. There is a commutative diagram 
\[
\xymatrix{
r^*F \ar[r] \ar[d] & r^*(T_n^\mathbf{O}F) \ar[d] \\
T_n^\mathbf{U}(r^*F) \ar[r] & T_n^\mathbf{U}(r^*T_n^\mathbf{O})
}
\]
in which in right--hand vertical map is a levelwise weak equivalence by Theorem \ref{poly for r}. Furthermore, the top horizontal map is an order $n$ orthogonal agreement, since $F$ is weakly $(\rho,n)$-polynomial and pre-realification preserves agreements, Lemma \ref{realification of an agreement}. The lower horizontal map is then a levelwise weak equivalence by Lemma \ref{agreement gives agreeing polynomials}. It follows that the map $r^*F \to T_n^\mathbf{U}(r^*F)$ is a levelwise weak equivalence if and only if the map $r^*F \to  r^*(T_n^\mathbf{O}F)$ is a levelwise weak equivalence.
\end{proof}

\begin{thm}\label{thm: agreeing tower after realification}
Let $F$ be a weakly polynomial reduced orthogonal functor. Then the unitary Taylor tower associated to $r^*F$  is equivalent to the pre-realification of the orthogonal Taylor tower associated to $F$, that is, for all $V \in \J_0^\mathbf{U}$, and all $n \geq 0$, there is a zig-zag of weak equivalences between the top and bottom rows of the following diagram, 
\[
\xymatrix{
r^*D_{n}^\mathbf{O}F(V) \ar[r]  \ar[d]^\simeq & r^*T_{n}^\mathbf{O}F(V) \ar[r] \ar[d]^\simeq & r^*T_{n-1}^\mathbf{O}F(V) \ar[d]^\simeq \\
\mathcal{F}(V) \ar[r]  & T_{n}^\mathbf{U}(r^*(T_{n}^\mathbf{O}F))(V) \ar[r] &T_{n-1}^\mathbf{U}(r^*(T_{n-1}^\mathbf{O}F))(V) \\
 D_{n}^\mathbf{U}(r^*F)(V) \ar[r] \ar[u]_\simeq  & T_{n}^\mathbf{U}(r^*F)(V) \ar[r] \ar[r] \ar[u]_\simeq & T_{n-1}^\mathbf{U}(r^*F)(V).  \ar[u]_\simeq
}
\]
where $\mathcal{F}(V)$ is the homotopy fibre of the map $T_{n}^\mathbf{U}(r^*(T_{n}^\mathbf{O}F))(V) \to T_{n-1}^\mathbf{U}(r^*(T_{n-1}^\mathbf{O}F))(V)$.
\end{thm}

\begin{proof}
There is a commutative diagram 
\[
\xymatrix@C+1cm{
r^*(T_n^\mathbf{O}F) \ar[r]^{\eta_n} \ar[d] & T_n^\mathbf{U}(r^*(T_n^\mathbf{O}F)) \ar[d] & T_n^\mathbf{U}(r^*F) \ar[l]_(.4){T_n^\mathbf{U}(r^*\eta_n)} \ar[d] \\
r^*(T_{n-1}^\mathbf{O}F) \ar[r]_{\eta_{n-1}}  & T_{n-1}^\mathbf{U}(r^*(T_{n-1}^\mathbf{O}F)) & T_{n-1}^\mathbf{U}(r^*F) \ar[l]^(.4){T_{n-1}^\mathbf{U}(r^*\eta_{n-1})} \\
}
\]
where the left hand horizontal maps are both weak equivalences by Theorem \ref{poly for r} and the right hand horizontal maps are both weak equivalences by Proposition \ref{prop: weak poly give agreeing approxs}. It follows that $r^*D_n^\mathbf{O}F$ is levelwise weakly equivalent to $D_n^\mathbf{U}(r^*F)$, and the result follows. 
\end{proof}

Combining Theorem \ref{thm: agreeing tower after realification} with Proposition \ref{homog r} we achieve the following corollary.

\begin{cor}
If $F$ is weakly polynomial, reduced and $\Theta_F^n$ is the spectrum associated to the homogeneous functor $D_n^\mathbf{O}F$, then $\U(n)_+ \wedge^\mathbb{L}_{\O(n)} \Theta_F^n$ is the spectrum associated to the homogeneous functor $D_n^\mathbf{U}(r^*F)$.
\end{cor}

%-----------------------------------------------------------------------
%-----------------------------------------------------------------------
\section{A complete model category comparison}\label{section: model categories}
%-----------------------------------------------------------------------
%-----------------------------------------------------------------------

Combining the model categories for orthogonal and unitary calculus produces the following diagram, Figure \ref{fig 1}, which gives a complete comparison between the orthogonal and unitary calculi.  The remained of this paper is devoted to demonstrating how this diagram commutes. In Section \ref{section: spectra} we consider the comparisons between the different categories of spectra used throughout the calculi, and show that the top portion of the diagram commutes. In Section \ref{section: intermediate categories}, we turn our attention to the intermediate categories for the calculi. Here we introduce two new categories, which act as intermediate categories between the standard intermediate categories. With these in place, we demonstrate how the middle portion of Figure \ref{fig 1} commutes. It is then only left to describe how the lower two pentagons of Figure \ref{fig 1} commute. We deal with this in Section \ref{section: homotopy categories}. This is considerably more complex since we are attempting to compose left and right Quillen functors with each other. Before turning our attention to proving that Figure \ref{fig 1} commutes, we give an example of how this diagram may be applied in practice. This will utilise the results of Sections \ref{section: spectra} and \ref{section: intermediate categories}, and especially the commutation results of Section \ref{section: homotopy categories}. 

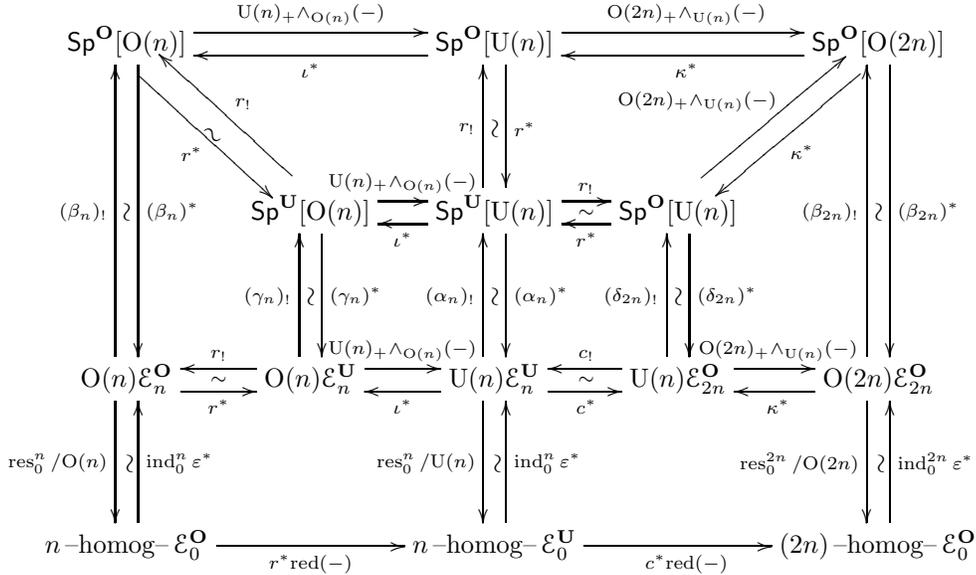
\begin{figure}[h]
\[
\xymatrix@C=1em@R=2em{
\s^\mathbf{O}[\O(n)] \ar@<1ex>[dddd]^{(\beta_n)^*}_\vsim	\ar@<1ex>[rr]^{\U(n)_+\wedge_{\O(n)}(-)}\ar@<-1.3ex>[ddr]_{r^*}				& 																												&		\s^\mathbf{O}[\U(n)]	 \ar@<1ex>[dd]^{r^*}_\vsim \ar@<1ex>[rr]^{\O(2n)_+\wedge_{\U(n)}(-)}\ar@<1ex>[ll]^{\iota^*}									&										&		\s^\mathbf{O}[\O(2n)] \ar@<1ex>[dddd]^{(\beta_{2n})^*}_{\vsim}	\ar@<1ex>[ll]^{\kappa^*}	 	\ar@<1ex>[ddl]^{\kappa^*}				\\
																										& 																												&																																									&										&											\\
																											& 	\s^\mathbf{U}[\O(n)]	 \ar@<1ex>[dd]^{(\gamma_n)^*}_\vsim \ar@<1ex>[r]^{\U(n)_+\wedge_{\O(n)}(-)} \ar@<-1.3ex>[uul]_{r_!}^\dsim			&		\s^\mathbf{U}[\U(n)] \ar@<1ex>[dd]^{(\alpha_n)^*}_\vsim 	\ar@<1ex>[uu]^{r_!} \ar@<1ex>[l]^{\iota^*}	\ar@<1ex>[r]^{r_!}_\sim												&	\s^\mathbf{O}[\U(n)] \ar@<1ex>[dd]^{(\delta_{2n})^*}_\vsim 	\ar@<1ex>[l]^{r^*}	\ar@<1ex>[uur]^{\O(2n)_+ \wedge_{\U(n)}(-)}								&											\\
																										& 																												&																																																	&										&											\\
\O(n)\E_n^\mathbf{O} \ar@<-1ex>[dd]_{\res_0^n/\O(n)}^\vsim \ar@<1ex>[uuuu]^{(\beta_n)_!}\ar@<-1ex>[r]_{r^*}^\sim	&		\O(n)\E_n^\mathbf{U} \ar@<1ex>[uu]^{(\gamma_n)_!} \ar@<1ex>[r]^{\U(n)_+\wedge_{\O(n)}(-)}\ar@<-1ex>[l]_{r_!}		& 		\U(n)\E_n^\mathbf{U} \ar@<-1ex>[dd]_{\res_0^n/\U(n)}^\vsim \ar@<1ex>[uu]^{(\alpha_n)_!} \ar@<1ex>[l]^{\iota^*} \ar@<-1ex>[r]_{c^*}^\sim											& 		\U(n)\E_{2n}^\mathbf{O} \ar@<1ex>[r]^{\O(2n)_+\wedge_{\U(n)}(-)} \ar@<-1ex>[l]_{c_!}\ar@<1ex>[uu]^{(\delta_{2n})_!}		&		\O(2n)\E_{2n}^\mathbf{O} \ar@<-1ex>[dd]_{\res_0^{2n}/\O(2n)}^\vsim \ar@<1ex>[uuuu]^{(\beta_{2n})_!}\ar@<1ex>[l]^{\kappa^*}				\\
																											& 																												&																																									&										&											\\
n\homog\E_0^\mathbf{O} \ar@<-1ex>[uu]_{\ind_0^n\varepsilon^*}\ar[rr]_{r^*\mathrm{red}(-)}											& 																												&		n\homog\E_0^\mathbf{U} \ar@<-1ex>[uu]_{\ind_0^n\varepsilon^*}	\ar[rr]_{c^*\mathrm{red}(-)}																	&										&		(2n)\homog\E_0^\mathbf{O}	 \ar@<-1ex>[uu]_{\ind_0^{2n}\varepsilon^*}		\\
}
\]
\caption{Model categories for orthogonal and unitary calculi}
\label{fig 1}
\end{figure}

The $n$-sphere functor interacts well with our comparisons as it is reduced. In fact the following example works just as well for $\J_n(V,-)$, $V \in \J^\mathbf{O}$.

\begin{example}
Let $n\mathbb{S}$ be the $n$-sphere from orthogonal calculus, i.e. the object $\J_n^\mathbf{O}(0,-)$ in  $n\homog\E_0^\mathbf{O}$. Under the Quillen equivalence between $n\homog\E_0^\mathbf{O}$ and $\O(n)\E_n^\mathbf{O}$, $n\mathbb{S}$ corresponds to $\O(n)_+ \wedge n\mathbb{S}$ in $\O(n)\E_n^\mathbf{O}$, which under the Quillen equivalence between $\s^\mathbf{O}[\O(n)]$ and $\O(n)\E_n^\mathbf{O}$ corresponds to $\O(n)_+ \wedge \mathbb{S}$, that is, 
\[\xymatrix@C+1cm{
\O(n)_+ \wedge \mathbb{S}
&
\ar@{|->}[l]_-{\mathds{L}{(\beta_n)}_!}
\O(n)_+ \wedge n\mathbb{S}
\ar@{|->}[r]^-{\mathds{L} \res_0^n/\O(n)} &
n\mathbb{S}
}\]
Applying (derived) change of group functor sends $\O(n)_+ \wedge \mathbb{S}$ to $\U(n)_+ \wedge \mathbb{S}$. As before, this is the stable $n$-th derivative of $n\mathbb{S}$, i.e.
\[\xymatrix@C+1cm{
\U(n)_+ \wedge \mathbb{S}
&
\ar@{|->}[l]_-{\mathds{L}{(\alpha_n\circ r)}_!}
\U(n)_+ \wedge n\mathbb{S}
\ar@{|->}[r]^-{\mathds{L} \res_0^n/\U(n)} &
n\mathbb{S}
}\]

It follows that $\mathds{R}r^*(n\mathbb{S}) \cong n\mathbb{S}$ in $\Ho(n\homog\E_0^\mathbf{U})$. Applying the (derived) change of group functor $\U(n)_+ \wedge \mathbb{S}$ corresponds to $\O(2n)_+ \wedge n\mathbb{S}$ in $\s^\mathbf{O}[\O(2n)]$. This is the stable $(2n)$-th derivative of $(2n)\mathbb{S}$, i.e.
\[
\xymatrix@C+1cm{
\O(2n)_+ \wedge \mathbb{S}
&
\ar@{|->}[l]_-{\mathds{L}{(\beta_{2n})}_!}
\O(2n)_+ \wedge {(2n)}\mathbb{S}
\ar@{|->}[r]^-{\mathds{L} \res_0^{2n}/\O(2n)} &
({2n})\mathbb{S}
}
\]

It follows that $c^*(n\mathbb{S}) \cong (2n)\mathbb{S}$, in $\Ho((2n)\homog\E_0^\mathbf{O})$ and $(rc)^*(n\mathbb{S}) \cong (2n)\mathbb{S}$ in $\Ho((2n)\homog\E_0^\mathbf{O})$. This is the functor calculus version of complexification followed by realification resulting in a vector space of twice the original dimension. 
\end{example}

%-----------------------------------------------------------------------
%-----------------------------------------------------------------------
\section{Comparisons of spectra}\label{section: spectra}
%-----------------------------------------------------------------------
%-----------------------------------------------------------------------
We have constructed a Quillen adjunction between the orthogonal and unitary $n$-homogeneous model structures. To give a complete comparison of the theories we must address the comparisons between the other two categories in the zig-zag of Quillen equivalences of Barnes and Oman \cite{BO13} and the author \cite{Ta19}. We start by addressing the relationship between the categories of spectra. For this, we recall the definitions and model structures involved.

\begin{definition}
For a compact Lie group $G$, the category $\s[G]$, is the category of $G$-objects in $\s$ and $G$-equivariant maps, that is, an object in $\s[G]$ is a continuous functor $ X \colon \J_1^\F \to \T$ such that there is a group homomorphism $G \to \Aut(X)$, where $\Aut(X)$ denotes the group of automorphism of $X$ in $\s$. We will refer to  $\s[G]$ as the category of naïve $G$-spectra.
\end{definition}

The category  $\s[G]$ comes with a levelwise and a stable model structure induced by the standard levelwise and stable model structures of $\s$. Importantly theses model structures are defined independently of group actions, we take levelwise weak equivalences, not levelwise weak equivalences on fixed points. 

\begin{lem}
There is a cellular proper and topological model structure on the category $\s[G]$ with the weak equivalences and fibrations defined levelwise. This model structure is called the projective model structure.
\end{lem}

\begin{lem}
There is a cofibrantly generated topological model structure on the category $\s[G]$ with, weak equivalences the $\pi_*$-isomorphisms, and fibrations the levelwise fibrations $f\colon \Theta \to \Psi$ such that the diagram 
\[
\xymatrix{
\Theta(V) \ar[r] \ar[d] & \Psi(V) \ar[d] \\
\Omega^W \Theta(V \oplus W) \ar[r] & \Omega^W \Psi(V \oplus W)
}
\]
is a homotopy pullback square for all $V, W \in \J_1$.
\end{lem}

The $\pi_*$-isomorphism rely on the model of spectra. For example, they are defined as 
\[
\pi_k(\Theta) = \colim_{q} \pi_{k+q}(\Theta(\R^q))
\]
for orthogonal spectra, and 
\[
\pi_k(\Theta) = \colim_{q} \pi_{k+2q}(\Theta(\C^q))
\]
for unitary spectra, with the difference coming from the fact that the commutative monoid $\mathbb{S}$ for unitary spectra only takes values on even dimensional spheres. 

\subsection{Change of group}

Let $H$ be a subgroup of a compact Lie group $G$. Then given a spectrum (in any chosen model) with an action of $G$, we can restrict through the subgroup inclusion $\iota \colon H \to G$ to given the spectrum an action of $H$. 

\begin{definition}
For a spectrum $\Theta$ with an action of $G$, let $\iota^* \Theta$ be the same spectrum $\Theta$ with an action of $H$ formed by forgetting structure through $\iota$.
\end{definition}

In detail, let $\I_\F$ be the category of $\F$-inner product subspaces of $\F^\infty$ with $\F$-linear isometric isomorphisms. For a spectrum $\Theta$ with $G$-action, the evaluations maps 
\[
\Theta_{U,V}\colon \I_\F(U, V) \to \T(\Theta(U), \Theta(V))
\]
are $G$-equivariant. We can apply $\iota^*$ to this, to give a map which is $H$-equivariant by forgetting structure, 
\[
\iota^*\Theta_{U,V} \colon \iota^* \I_\F(U,V) \to \iota^*\T(\Theta(U), \Theta(V))= \T(\iota^*\Theta(U), \iota^*\Theta(V)).
\]

This functor has a left adjoint $G_+ \wedge_{H} - \colon \s[H] \to \s[G]$, given on an object $\Theta$ of $\s[H]$ by 
\[
(G_+\wedge_{H} \Theta)(V) = G_+ \wedge_{H} \Theta(V), 
\]
compare \cite[Proposition VI.2.3]{MM02}.

\begin{prop}\label{change of group QA}
The adjoint pair
\[
\adjunction{G_+ \wedge_{H} -}{\s[H]}{\s[G]}{\iota^*}
\]
is a Quillen adjunction. 
\end{prop}
\begin{proof}
This follows immediately from noting that the $\pi_*$-isomorphisms and q-fibrations are defined independently of the group action.
\end{proof}

For our particular groups of interest, we achieve the following.

\begin{cor}
The adjoint pair
\[
\adjunction{\U(n)_+ \wedge_{\O(n)} -}{\s[\O(n)]}{\s[\U(n)]}{\iota^*}
\]
is a Quillen adjunction. 
\end{cor}

\begin{cor}
The adjoint pair
\[
\adjunction{\O(2n)_+ \wedge_{\U(n)} -}{\s[\U(n)]}{\s[\O(2n)]}{\kappa^*}
\]
is a Quillen adjunction. 
\end{cor}

\subsection{Change of model} There is also a change of model subtly involved in the theory. This was proven in \cite{Ta19}, where the author produced a Quillen equivalence between orthogonal and unitary spectra. 

\begin{prop}[{\cite[Theorem 6.4]{Ta19}}]
The adjoint pair
\[
\adjunction{r_!}{\s^\mathbf{U}}{\s^\mathbf{O}}{r^*}
\]
is a Quillen equivalence. 
\end{prop}

In \cite[Corollary 6.5]{Ta19}, the author applied the above Quillen equivalence to a Quillen equivalence between the categories of $\U(n)$-objects in both models for the stable homotopy category. The same is true - with analogous proof - for $\O(n)$-objects in both model for the stable homotopy category. 

\begin{cor}
The adjoint pair 
\[
\adjunction{r_!}{\s^\mathbf{U}[\O(n)]}{\s^\mathbf{O}[\O(n)]}{r^*}
\]
is a Quillen equivalence.
\end{cor}

The change of group and change of model are compatible in the following sense. Let $H \leq G$ act as notation for either the subgroup inclusion $\O(n) \leq \U(n)$ or the inclusion $\U(n) \leq \O(2n)$. 

\begin{lem}\label{lem: spectra and change of group}
The diagram
\[
\xymatrix@C+1cm{
\s^\mathbf{O}[H] 	\ar@<1ex>[r]^{G_+ \wedge_{H}(-)}	\ar@<1ex>[d]^{r^*}	&		\s^\mathbf{O}[G] 	\ar@<1ex>[l]^{\iota^*}	\ar@<1ex>[d]^{r^*}\\
\s^\mathbf{U}[H]	\ar@<1ex>[r]^{G_+ \wedge_{H}(-)}	\ar@<1ex>[u]^{r_!}		&		\s^\mathbf{U}[G]	\ar@<1ex>[l]^{\iota^*}	\ar@<1ex>[u]^{r_!}
}
\]
commutes up to natural isomorphism.
\end{lem}
\begin{proof}
It suffices to show that the diagram of right adjoints commute. Indeed, 
\[
\iota^*((r^*\Theta)(V)) = \iota^* \Theta(V_\R) = (\iota^*\Theta)(V_\R) = r^*((\iota^*\Theta)(V)). \qedhere
\]
\end{proof}

%-----------------------------------------------------------------------
%-----------------------------------------------------------------------
\section{Comparing the intermediate categories}\label{section: intermediate categories}
%-----------------------------------------------------------------------
%-----------------------------------------------------------------------
To achieve the correct correspondences between $\U(n)\E_n^\mathbf{U}$ and $\O(n)\E_n^\mathbf{O}$ we introduce two new intermediate categories via the inclusion of subgroups $\iota\colon \O(n) \hookrightarrow \U(n)$ and $\kappa\colon \U(n) \hookrightarrow \O(2n)$. In our consideration of the comparisons between the categories of spectra, the order in which we changed the group and changed the model was unimportant, since the indexing categories, $\J_1^\mathbf{O}$ and $\J_1^\mathbf{U}$, are equipped with the trivial action. However, for the intermediate categories, the diagram categories $\J_n$ have a non-trivial action of $\Aut(n)$, hence the order in which one changes group and changes model is important. This section gives the correct method for such comparisons. 

\begin{definition}
Define $\O(n)\E_n^\mathbf{U}$ to be the category of $\O(n)\T$-enriched functors from $\iota^*\J_n^\mathbf{U}$ to $\O(n)\T$ where $\iota^*\J_n^\mathbf{U}$ is an $\O(n)\T$-enriched category obtained from $\J_n^\mathbf{U}$ by forgetting structure through the subgroup inclusion $\iota\colon \O(n) \to \U(n)$. Similarly define $\U(n)\E_{2n}^\mathbf{O}$ to be the category of $\U(n)\T$-enriched functors from $\kappa^*\J_{2n}^\mathbf{O}$ to $\U(n)\T$ where $\kappa^*\J_{2n}^\mathbf{O}$ is an $\U(n)\T$-enriched category obtained from $\J_{2n}^\mathbf{O}$ by forgetting structure through the subgroup inclusion $\kappa\colon \U(n) \to \O(2n)$. 
\end{definition}

These categories also come with projective and stable model structures constructed analogously to those of Proposition \ref{prop: n-stable model structure}. These new intermediate categories will now act as intermediate categories between the standard intermediate categories of orthogonal and unitary calculus. Further, the new intermediate categories equipped with their $n$-stable model structure are Quillen equivalent to spectra with an appropriate group action. The proofs of the following two results follow similarly to \cite[Proposition 8.3]{BO13} and \cite[Theorem 6.8]{Ta19}.

\begin{prop}
There is a Quillen equivalence 
\[	\adjunction{(\gamma_n)_!}{\O(n)\E_n^\mathbf{U}}{\s^\mathbf{U}[\O(n)]}{(\gamma_n)^*}
\]
with $(\gamma_n)^* \Theta (V) = \Theta(\C^n \otimes_\C V)$, and $(\gamma_n)_!$ is the left Kan extension along $\gamma_n$. 
\end{prop}

\begin{prop}
	There is a Quillen equivalence 
	\[	\adjunction{(\delta_{2n})_!}{\U(n)\E_{2n}^\mathbf{O}}{\s^\mathbf{O}[\U(n)]}{(\delta_{2n})^*}
	\]
	with $(\delta_{2n})^* \Theta (V) = \Theta(\R^{2n} \otimes_\R V)$, and $(\delta_{2n})_!$ is the left Kan extension along $\delta_{2n}$. 
\end{prop}

\subsection{Change of group}
Let $E \in \U(n)\E_n^\mathbf{U}$, then $E$ is defined by $\U(n)$-equivariant structure maps of the form
\[
E_{U,V} \colon \J_n^\mathbf{U}(U,V) \to \T(E(U), E(V)).
\]
Forgetting structure through $\iota \colon\O(n) \to \U(n)$ yields an $\O(n)$-equivariant map 
\[
\iota^*E_{U,V} \colon \iota^*\J_n^\mathbf{U}(U,V) \to \iota^*\T(E(U), E(V))= \T(\iota^*E(U), \iota^* E(V)).
\]
This induces a functor $\iota^*\colon \U(n)\E_n^\mathbf{U} \to \O(n)\E_n^\mathbf{U}$, which has a left adjoint $\U(n)_+ \wedge_{\O(n)}(-)$ given by 
\[
(\U(n)_+ \wedge_{\O(n)} E)(V) = \U(n)_+ \wedge_{\O(n)} E(V),
\]
with structure maps 
\[
S^{nW} \wedge \U(n)_+ \wedge_{\O(n)} E(V) \cong \U(n)_+ \wedge_{\O(n)} (\iota^*S^{nW} \wedge E(V)) \to \U(n)_+ \wedge E(V \oplus W),
\]
where the isomorphism follows from \cite[Proposition V.2.3]{MM02}.

Completely analogous constructions for the subgroup inclusion $i\colon \U(n) \to \O(2n)$ yields an adjoint pair
\[
\adjunction{\O(2n)_+ \wedge_{\U(n)}(-)}{\U(n)\E_{2n}^\mathbf{O}}{\O(2n)\E_{2n}^\mathbf{O}}{\kappa^*}.
\]

\begin{lem}
The adjoint pair
\[
\adjunction{\U(n)_+ \wedge_{\O(n)}(-)}{\O(n)\E_{n}^\mathbf{U}}{\U(n)\E_{n}^\mathbf{U}}{\iota^*}.
\]
is a Quillen adjunction.
\end{lem}
\begin{proof}
The levelwise fibrations, levelwise weak equivalences and $n\pi_*$-isomorphisms are defined independently of group actions. It follows that $\iota^*$ preserves these. 
\end{proof}

This results in a square of adjoint functors,
\[
\xymatrix@C+1cm{
\s^\mathbf{U}[\O(n)]  \ar[r]<1.1ex>^{\U(n)_+ \wedge_{\O(n)} -} \ar[d]<1.1ex>^{(\gamma_n)^*} & \s^\mathbf{U}[\U(n)]  \ar[l]<1.1ex>^{\iota^*} \ar[d]<1.1ex>^{(\alpha_n)^*}\\
\O(n)\E_n^\mathbf{U} \ar[r]<1.1ex>^{\U(n)_+ \wedge_{\O(n)}-} \ar[u]<1.1ex>^{{(\gamma_n)}_!} & \U(n)\E_n^\mathbf{U} \ar[l]<1.1ex>^{\iota^*} \ar[u]<1.1ex>^{{(\alpha_n)}_!}.
}
\]

\begin{lem}\label{square 10 commutes}
The above diagram commutes up to natural isomorphism. 
\end{lem}
\begin{proof}
Let $X$ be a unitary spectrum with an action of $\U(n)$. Then 
\[
(\iota^*\alpha_n^*)(X)(V) = \iota^* X(nV) = \gamma_n^* (\iota^*X)(V). 
\]
The result then follows immediately. Note that the functor $\iota^*$ restricts the group actions in a compatible way. The restricted action of $\O(n)$ on $X(nV)$ is $\iota^*(X(\sigma \otimes V) \circ X_{\sigma(nV)})$ where $\sigma \in \U(n)$. This is equivalent to the action $X(\iota^*(\sigma) \otimes V) \circ X_{\iota^*(\sigma)(nV)}$ since $X(\sigma \circ V)$ and $X_{\sigma(nV)}$ commute. This action is precisely the action we get from first restricting the action and then applying $\gamma_n^*$.
\end{proof}

Similarly we obtain the following result. 

\begin{lem}
The adjoint pair
\[
\adjunction{\O(2n)_+ \wedge_{\U(n)}(-)}{\U(n)\E_{2n}^\mathbf{O}}{\O(2n)\E_{2n}^\mathbf{O}}{\kappa^*}.
\]
is a Quillen adjunction.
\end{lem}

This results in a square of adjoint functors,
\[
\xymatrix@C+1cm{
\s^\mathbf{O}[\U(n)]  \ar[r]<1.1ex>^{\O(2n)_+ \wedge_{\U(n)} -} \ar[d]<1.1ex>^{(\delta_{2n})^*} & \s^\mathbf{O}[\O(2n)]  \ar[l]<1.1ex>^{\kappa^*} \ar[d]<1.1ex>^{(\beta_{2n})^*}\\
\U(n)\E_{2n}^\mathbf{O} \ar[r]<1.1ex>^{\O(2n)_+ \wedge_{\U(n)}-} \ar[u]<1.1ex>^{{(\delta_{2n})}_!} & \O(2n)\E_{2n}^\mathbf{O} \ar[l]<1.1ex>^{\kappa^*} \ar[u]<1.1ex>^{{(\beta_{2n})}_!}.
}
\]

\begin{lem}\label{change of group for c}
The above diagram commutes up to natural isomorphism. 
\end{lem}
\begin{proof}
Let $X$ be an orthogonal spectrum with an action of $\O(2n)$. Then 
\[
(\kappa^*\beta_{2n}^*)(X)(V) = \kappa^* X(nV) = \delta_{2n}^* (\kappa^*X)(V). 
\]
The result then follows immediately. The group actions restrict in a compatible way as in Lemma \ref{square 10 commutes}.
\end{proof}

\subsection{Change of model through realification} 
We define a realification functor $r\colon \J_n^\mathbf{U} \to \J_n^\mathbf{O}$. This functor induces a right Quillen functor between $\O(n)\E_n^\mathbf{U}$ and $\O(n)\E_n^\mathbf{O}$. 

On objects, let $r$ be given by forgetting the complex structure, i.e., $\C^k \longmapsto \R^{2k}$. Morphisms in $ \J_n^\mathbf{U}$ are given in terms of the Thom space of the vector bundle 
\[
\gamma_n^\mathbf{U} (V,W) = \{ (f,x) \ : \ f \in \J_0^\mathbf{U}(V,W), \ x \in \C^n \otimes_\C f(V)^\perp\}
\]
over the space of linear isometries $\J_0^\mathbf{U}(V,W)$. We then define realification on a pair $(f,x)$ by 
\[
r(f,x) = (f_\R, rx)
\]
where $f_\R \in \J_0^\mathbf{O}(V_\R, W_\R)$ and $rx$ is the image of $x$ under the  $\R$--linear isomorphism
\[
\C^n \otimes_\C f(V)^\perp \to \R^n \otimes_\R (f_\R)(V_\R)^\perp.
\]
%This map is the composite 
%\[
%\C^n \otimes_\C (W-f(V)) \cong \bigoplus_{i =1}^n f(V)^\perp \xrightarrow{\bigoplus_{i=1}^n r} \bigoplus_{i=1}^n (f_\R)(V_\R)^\perp \cong \R^n \otimes_\R (W_\R- (f_\R)(V_\R)),
%\] 
%where $r$ is the standard realification map on vector spaces. It is not hard to check that $r$ gives a well defined map $f(V)^\perp \to (f_\R)(V_\R)^\perp$. 

Restricting the $\U(n)$ action on $\J_n^\mathbf{U}$ to an action of $\O(n)$ through the subgroup inclusion $\iota\colon \O(n) \hookrightarrow \U(n)$, induces a functor 
\[
r\colon \iota^* \J_n^\mathbf{U} \to \J_n^\mathbf{O}
\]
and precomposition defines a functor 
\[
r^* \colon \O(n)\E_n^\mathbf{O} \to \O(n)\E_n^\mathbf{U}. 
\]

To see that $r^*$ is well defined, we check that the map 
\[
(r^*F) \colon \iota^*\J_n^\mathbf{U}(V,W) \to \T((r^*F)(V), (r^*F)(V)) = \T(F(V_\R), F(W_\R))
\]
is $\O(n)$-equivariant where $F \in \O(n)\E_n^\mathbf{O}$. Indeed, let $(f,x) \in \iota^*\J_n^\mathbf{U}(V,W)$ and $\sigma \in \O(n)$, 
\[
(r^*F)(\sigma(f,x)) = (r^*F)(f, \iota(\sigma)(x)) = F(f_\R, r(\iota(\sigma)(x))).
\]
For $W$ a complex vector space, the restricted action of $\U(n)$ to $\O(n)$ on $\C^n \otimes_\C W$ is compatible with the $\O(n)$-action on $\R^n \otimes_\R rW$, hence $r(\iota^*(\sigma)(x)) = \sigma (rx)$, and the above becomes 
\[
F(f_\R, r(\iota^*(\sigma)(x))) = F(f_\R, \sigma (rx)) = \sigma (F(f_\R, rx)) = \sigma ( (r^*F)(f,x)). 
\]
It follows that the required map is $\O(n)$-equivariant and hence $r^*F$ is a well defined object of $\O(n)\E_n^\mathbf{U}$. 

The structure maps of $r^*F$ are given by iterating the structure maps of $F$;
\[
S^{2n} \wedge (r^*F)(\C^k) \xrightarrow{=} S^{2n} \wedge F(\R^{2k}) \xrightarrow{\sigma^2} F(\R^{2k+2}) \xrightarrow{=} (r^*F)(\C^{k+1}),
\]
where $\sigma \colon S^n \wedge F(\R^k) \to F(\R^{k+1})$ is the structure map of $F$. As $r^*$ is defined by precomposition it has a natural left adjoint, $r_!$ given by the left Kan extension along $r$.

\begin{lem}\label{lem: realification as QF}
The functor $r^*\colon \O(n)\E_n^\mathbf{O} \to \O(n)\E_n^\mathbf{U}$ is a right Quillen functor. 
\end{lem}
\begin{proof}
By definition on objects, $r^*$ preserves all levelwise weak equivalences and all levelwise fibrations. The compatibility with $r^*$ and the structure maps shows that $r^*$ preserves fibrant objects. 
\end{proof}

This comparison produces a diagram of adjoint functors 
\[
\xymatrix@C+1cm{
\s^\mathbf{O}[\O(n)]  \ar[r]<-1.1ex>_{r^*} \ar[d]<1.1ex>^{(\beta_n)^*} & \s^\mathbf{U}[\O(n)]  \ar[l]<-1.1ex>_{r_!} \ar[d]<1.1ex>^{(\gamma_n)^*}\\
\O(n)\E_n^\mathbf{O} \ar[r]<-1.1ex>_{r^*} \ar[u]<1.1ex>^{{(\beta_n)}_!} & \O(n)\E_n^\mathbf{U} \ar[l]<-1.1ex>_{r_!} \ar[u]<1.1ex>^{{(\gamma_n)}_!}.
}
\]

\begin{lem}\label{square 9 commutes}
The above diagram commutes up to natural isomorphism. 
\end{lem}
\begin{proof}
Consider the diagram of enriched categories, 
\[
\xymatrix@C+1cm{
\J_1^\mathbf{O}   & \J_1^\mathbf{U} \ar[l]_r \\
\J_n^\mathbf{O} \ar[u]^{\beta_n} & \J_n^\mathbf{U} \ar[l]_r \ar[u]_{\gamma_n}
}
\]
It is clear from the definition of these functors that the diagram commutes on objects up to natural isomorphism.  Now on morphisms, take $(f,x) \in \J_n^\mathbf{U}$. Then 
\[
r(\gamma_n(f,x)) = r((\C^n \otimes f, x) = (\R^n \otimes f_\R, rx) = \beta_n((f_\R, rx)) = \beta_n(r(f,x)). 
\]
It follows that $r \gamma_n = \beta_n r$. Since the right adjoints in the required diagram are defined in terms of precomposition the result follows.  Note that the group actions are also compatible since the unitary $\gamma_n$ has been restricted to $\O(n)$ actions. 
\end{proof}

\subsection{Change of model through complexification}
Define a complexification functor $c\colon \J_{2n}^\mathbf{O} \to \J_{n}^\mathbf{U}$, given on objects by $cV = \C \otimes V$, and on morphisms by sending $(f,x) \in \J_{2n}^\mathbf{O}(V,W)$ to $(\C \otimes f, cx) \in \J_n^\mathbf{U}(cV, cW)$, where $cx$ is the image of $x$ under the composition of isomorphisms,
\[
\R^{2n} \otimes_\R \coker(f) \xrightarrow[\cong]{\varphi_1} \C^n \otimes_\R \coker(f) \xrightarrow[\cong]{\varphi_2} \C^n \otimes_\C \C \otimes_\R \coker(f) \xrightarrow[\cong]{\varphi_3} \C^n \otimes_\C \coker(\C \otimes f). 
\]
where 
\[
\begin{split}
\varphi_1 (r_1, \cdots, r_{2n}, f(v)) &= (r_1 +ir_2, \cdots r_{2n-1} +ir_{2n}, f(v)); \\
\varphi_2(c_1, \cdots, c_n, f(v)) &= (c_1, \cdots, c_n, 1, f(v)); \ \  \text{and}  \\
\varphi_3(c_1, \cdots, c_n, c, f(v)) &= (c_1, \cdots, c_n, (\C \otimes f)(c \otimes v)). \\
\end{split}
\]

Restricting from $\O(2n)$ to $\U(n)$ through the subgroup inclusion $\kappa \colon \U(n) \hookrightarrow \O(2n)$ gives a functor 
\[
c \colon \kappa^*\J_{2n}^\mathbf{O} \to \J_n^\mathbf{U},
\]
and precomposition defines a functor 
\[
c^* \colon \U(n)\E_n^\mathbf{U} \to \U(n)\E_{2n}^\mathbf{O}. 
\]

This functor is well defined as for $X \in \U(n)\E_n^\mathbf{U}$ the map 
\[
c^*X \colon\kappa^* \J_{2n}^\mathbf{O}(V,W) \to \T((c^*X)(V), (c^*X)(W))= \T(X (\C \otimes V), X(\C \otimes W))
\]
is $\U(n)$-equivariant. Indeed, for $\sigma \in \U(n)$, 
\[
\begin{split}
(c^*X)(\sigma(f,x)) = (c^*X)(f, \sigma x) &= X(\C \otimes f, c(\kappa^*(\sigma) x)) = X(\C \otimes f, \sigma(cx)) \\ &= X(\C \otimes f, \sigma x) = \sigma ((c^*X)(f,x)).
\end{split}
\]

The structure maps of $c^*X$ are induced by those of $X$, i.e., 
\[
S^{2n} \wedge (c^*X)(V) \xrightarrow{=} S^{2n} \wedge X(\C \otimes V) \xrightarrow{\sigma} X((\C \otimes V) \oplus \C) \xrightarrow{=} (c^*X)(V \oplus \R)
\]
where $\sigma\colon S^{2n} \wedge X(W) \to (W \oplus \C)$ is the structure map of $X \in \U(n)\E_{2n}^\mathbf{O}$. 

The complexification functor $c^*$ has a left adjoint, $c_!$ given by the left Kan extension along $c$. We obtain a similar result to the case of realification, Lemma \ref{lem: realification as QF}.

\begin{lem}
The functor $c^* \colon  \U(n)\E_n^\mathbf{U} \to \U(n)\E_{2n}^\mathbf{O}$ is a right Quillen functor.
\end{lem}

This produced a diagram of adjoint functors
\[
\xymatrix@C+1cm{
\s^\mathbf{U}[\U(n)]  \ar[r]<1.1ex>^{r_!} \ar[d]<1.1ex>^{(\alpha_n)^*} & \s^\mathbf{O}[\U(n)]  \ar[l]<1.1ex>^{r^*} \ar[d]<1.1ex>^{(\delta_{2n})^*}\\
\U(n)\E_n^\mathbf{U} \ar[r]<-1.1ex>_{c^*} \ar[u]<1.1ex>^{{(\alpha_n)}_!} & \U(n)\E_{2n}^\mathbf{O} \ar[l]<-1.1ex>_{c_!} \ar[u]<1.1ex>^{{(\delta_{2n})}_!}.
}
\]

\begin{lem}\label{intermediate categories and spectra for c}
The above diagram commutes up to natural isomorphism.
\end{lem}
\begin{proof}
Consider the diagram of enriched categories
\[
\xymatrix@C+1cm{
\J_{1}^\mathbf{U} \ar[r]^{r}   & \J_1^\mathbf{O} \\
\J_{n}^\mathbf{U}   \ar[u]^{\alpha_n} & \J_{2n}^\mathbf{O} \ar[u]_{\delta_{2n}} \ar[l]^{c}.
}
\]
It is clear that this diagram commutes up to natural isomorphism on objects. On morphisms, let $(f,x) \in \J_{2n}^\mathbf{O}(V,W)$. Then 
\[
\begin{split}
r(\alpha_n(c(f,x))) &\cong r(\alpha_n(\C \otimes f, cx))=r (\C^n \otimes_\C \C \otimes_\R f, cx) \\ &\cong r(\C^{n}\otimes f, cx) = (\R^{2n} \otimes f_\R, rcx) \\ &\cong \delta_{2n}(f,x),
\end{split}
\]
where $rcx = x$.
%since if $x$ is of the form $(r_1, \cdots, r_{2n}, f(v))$,
%\[
%cx = (r_1 +ir_2, \cdots r_{2n-1} +ir_{2n}, (\C \otimes f)(1 \otimes v))
%\]
%and hence 
%\[
%\begin{split}
%rcx &= r((r_1 +ir_2, \cdots r_{2n-1} +ir_{2n}, (\C \otimes f)(1 \otimes v))) \\
%      &=  (r_1, r_2, \cdots, r_{2n-1},r_{2n}, f(v)). \\
%\end{split}
%\]
It follows that $r\alpha_nc \cong \delta_{2n}$. As the right adjoints of the required diagram are defined in terms of precomposition, the result follows. The group actions are compatible by a similar argument to Lemma \ref{square 9 commutes}.
\end{proof}

%-----------------------------------------------------------------------
%-----------------------------------------------------------------------
\section{A homotopy category level comparison}\label{section: homotopy categories}
%-----------------------------------------------------------------------
%-----------------------------------------------------------------------

We have shown previously that all but the bottom pentagons of Figure \ref{fig 1} commute. Moreover, since all of the commutation results for the sub-diagrams - excluding the lower pentagons - involve composing left (resp. right) Quillen functors with left (resp. right) Quillen functors those sub-diagrams commute on the homotopy category level. Hence, the only sections of Figure \ref{fig 1} left to consider are the lower pentagons. The commuting of diagrams of adjoint functors (up to natural isomorphism) means that the respective diagrams of left and right adjoints commute. These pentagons are built from a mixture of left and right adjoints, so we must address how it commutes in a different manner.

\begin{lem}
The diagram 
\[
\xymatrix@C+1cm{
 \O(n)\E_n^\mathbf{O}  \ar[d]_{\res_0^n/\O(n)}  \ar[r]^{r^*} &  \O(n)\E_n^\mathbf{U} \ar[r]^{\U(n)_+\wedge_{\O(n)}  (-)}  &  \U(n)\E_n^\mathbf{U} \ar[d]^{\res_0^n/\U(n)} \\
n\homog\E_0^\mathbf{O}  \ar[rr]_{r^*} & & n\homog\E_0^\mathbf{U} 
}
\]
commutes up to natural isomorphism. 
\end{lem}
\begin{proof}
Consider the diagram of enriched categories 
\[
\xymatrix@C+1cm{
\J_n^\mathbf{O} & \ar[l]_{r}   \J_n^\mathbf{U} \\
\J_0^\mathbf{O}  \ar[u]^{i_0^n} & \ar[l]^{r}  \J_0^\mathbf{U} \ar[u]_{i_0^n} 
}
\]
where $i_0^n$ is the identity on objects and $f \longmapsto (f,0)$ on morphisms. This diagram clearly commutes on objects and morphisms. 
Let $X \in \O(n)\E_n^\mathbf{O}$. Then 
\[
\begin{split}
\res_0^n (\U(n)_+ \wedge_{\O(n)} (r^*X))/\U(n) &= (\U(n)_+ \wedge_{\O(n)} (X \circ r \circ i_0^n))/\U(n)  \\
								&\cong (X \circ r \circ i_0^n)/\O(n) \cong (X \circ i_0^n \circ r)/\O(n) \\
								&=r^*((\res_0^n X)/\O(n))
\end{split}
\]
where the first isomorphism comes from the fact that for any $\O(n)$-space $Y$, $(\U(n)_+ \wedge_{\O(n)} Y)/\U(n) \cong Y/\O(n)$, and the second isomorphism follows from the commutation of the above diagram of enriched categories. 
\end{proof}

\begin{lem}
The diagram
\[
\xymatrix@C+1cm{
\U(n)\E_n^\mathbf{U} \ar[r]^{c^*}  \ar[d]_{\res_0^n/\U(n)} &   \U(n)\E_{2n}^\mathbf{O} \ar[r]^{\O(2n)_+\wedge_{\U(n)}  (-)} & \O(2n)\E_{2n}^\mathbf{O} \ar[d]^{\res_0^{2n}/\O(2n)} \\
 n\homog\E_0^\mathbf{U}  \ar[rr]_{c^*}& & (2n)\homog\E_0^\mathbf{O} \\
 }
\]
commutes up to natural isomorphism.
\end{lem}
\begin{proof}
This proof follows similarly to the above, starting with the diagram of enriched categories
\[
\xymatrix@C+1cm{
\J_{n}^\mathbf{U} & \ar[l]_{c}  \J_{2n}^\mathbf{O}  \\
\J_0^\mathbf{U} \ar[u]^{i_0^{n}}&  \ar[l]^{c}  \J_0^\mathbf{O} \ar[u]_{i_0^{2n}}
}
\]
which commutes. 
\end{proof}

These squares are built using alternating left and right adjoints, hence no clean model category commutation is possible. We start with a larger diagram of homotopy categories and then restrict to our required diagram. On the homotopy category level we obtain the following result. 

\begin{lem}\label{big dia for r}
The following diagram of homotopy categories
\[
\xymatrix@C+1cm{
\Ho(\s^\mathbf{O}[\O(n)]) \ar[r]^{\U(n)_+ \wedge_{\O(n)}^\mathds{L} (-)} \ar[d]_{\mathds{R}(\beta_n)^*} & \Ho(\s^\mathbf{O}[\U(n)]) \ar[r]^{\mathds{R}r^*} & \Ho(\s^\mathbf{U}[\U(n)]) \ar[d]^{\mathds{R}(\alpha_n)^*} \\
\Ho(\O(n)\E_n^\mathbf{O}) \ar[d]_{\mathds{L}\res_0^n/\O(n)} & & \Ho(\U(n)\E_n^\mathbf{U}) \ar[d]^{\mathds{L}\res_0^n/\U(n)} \\
\Ho(n\homog\E_0^\mathbf{O}) \ar[rr]_{\mathds{R}r^*\mathrm{red}(-)} & & \Ho(n\homog\E_0^\mathbf{U}) 
}
\]
commutes up to natural isomorphism. 
\end{lem}
\begin{proof}
By the zig-zag of Quillen equivalences, \cite[Proposition 8.3 and Theorem 10.1]{BO13} the composite
\[
\mathds{L}\res_0^n/\O(n) \circ \mathds{R}(\beta_n)^*
\]
applied to an orthogonal spectrum $\Theta$ with an action of $\O(n)$, is levelwise weakly equivalent to the functor $F$, defined by the formula
\[
V \longmapsto \Omega^\infty[(S^{\R^n \otimes V} \wedge \Theta)_{h\O(n)}].
\]
This functor is $n$--homogeneous, hence also reduced, so $\mathds{R}r^*\mathrm{red}(F)$ is levelwise weakly equivalent to $\mathbb{R}r^*F$. The zig-zag of Quillen equivalences from unitary calculus, \cite[Theorems 6.8 and 7.5]{Ta19}, together with inflating $\Theta$ to a spectrum with an action of $\U(n)$ gives a similar characterisation in terms of an  $n$-homogeneous functor. The result then follows by Proposition \ref{homog r}. 
\end{proof}

A similar result holds true for similar diagram on the right of Figure \ref{fig 1}, utilising Lemma \ref{homog for c}, rather than Proposition \ref{homog r}. 

\begin{lem}\label{big dia for c}
The following diagram of homotopy categories
\[
\xymatrix@C+1cm{
\Ho(\s^\mathbf{U}[\U(n)]) \ar[r]^{\mathds{L}r_!} \ar[d]_{\mathds{R}(\alpha_n)^*} & \Ho(\s^\mathbf{O}[\U(n)]) \ar[r]^{\O(2n)_+ \wedge_{\U(n)}^\mathds{L} (-)} & \Ho(\s^\mathbf{O}[\O(2n)]) \ar[d]^{\mathds{R}(\beta_{2n})^*} \\
\Ho(\U(n)\E_n^\mathbf{U}) \ar[d]_{\mathds{L}\res_0^n/\U(n)} & & \Ho(\O(2n)\E_n^\mathbf{O}) \ar[d]^{\mathds{L}\res_0^n/\O(2n)} \\
\Ho(n\homog\E_0^\mathbf{U}) \ar[rr]_{\mathds{R}c^*\mathrm{red}(-)} & & \Ho((2n)\homog\E_0^\mathbf{O}) 
}
\]
commutes up to natural isomorphism. 
\end{lem}

\begin{cor}\label{other big diagram for r}
The following diagram of homotopy categories
\[
\xymatrix@C+1cm{
\Ho(\s^\mathbf{O}[\O(n)]) \ar[r]^{\U(n)_+ \wedge_{\O(n)}^\mathds{L} (-)}  & \Ho(\s^\mathbf{O}[\U(n)]) \ar[r]^{\mathds{R}r^*} & \Ho(\s^\mathbf{U}[\U(n)]) \ar[d]^{\mathds{R}(\alpha_n)^*} \\
\Ho(\O(n)\E_n^\mathbf{O}) \ar[u]^{\mathds{L}{(\beta_n)}_!} \ar[d]_{\mathds{L}\res_0^n/\O(n)} & & \Ho(\U(n)\E_n^\mathbf{U}) \ar[d]^{\mathds{L}\res_0^n/\U(n)} \\
\Ho(n\homog\E_0^\mathbf{O}) \ar[rr]_{\mathds{R}r^*\mathrm{red}(-)} & & \Ho(n\homog\E_0^\mathbf{U}) 
}
\]
commutes up to natural isomorphism. 
\end{cor}
\begin{proof}
By Lemma \ref{big dia for r}, there is a natural isomorphism
\[
\mathds{R}r^*\mathrm{red}(-) \circ \mathds{L}(\res_0^n/\O(n)) \circ \mathds{R}(\beta_n)^* \cong \mathds{L}(\res_0^n/\U(n)) \circ \mathds{R}(\alpha_n)^* \circ \mathds{R}r^* \circ (\U(n)_+ \wedge_{\O(n)}^\mathds{L}(-)).
\]
By the equivalence of the homotopy categories of $\s^\mathbf{O}[\O(n)]$ and $\O(n)\E_n^\mathbf{O}$ we have that $\mathds{R}(\beta_n)^* \circ \mathds{L}{(\beta_n)}_! \cong \mathds{1}$. It follows that
\[
\begin{split}
&\mathds{R}r^*\mathrm{red}(-) \circ \mathds{L}(\res_0^n/\O(n)) \circ \mathds{R}(\beta_n)^*\circ \mathds{L}{(\beta_n)}_!  \cong \mathds{R}r^* \mathrm{red}(-) \circ \mathds{L}(\res_0^n/\O(n)) \\
&\cong \mathds{L}(\res_0^n/\U(n)) \circ \mathds{R}(\alpha_n)^* \circ \mathds{R}r^* \circ  (\U(n)_+ \wedge_{\O(n)}^\mathds{L}(-))  \circ \mathds{L}{(\beta_n)}_! . \qedhere
\end{split}
\]
\end{proof}

\begin{cor}\label{other big diagram for c}
The following diagram of homotopy categories
\[
\xymatrix@C+1cm{
\Ho(\s^\mathbf{U}[\U(n)]) \ar[r]^{\mathds{L}r_!} & \Ho(\s^\mathbf{O}[\U(n)]) \ar[r]^{\O(2n)_+ \wedge_{\U(n)}^\mathds{L} (-)} & \Ho(\s^\mathbf{O}[\O(2n)]) \ar[d]^{\mathds{R}(\beta_{2n})^*} \\
\Ho(\U(n)\E_n^\mathbf{U})  \ar[u]^{\mathds{L}{(\alpha_n)}_!}  \ar[d]_{\mathds{L}\res_0^n/\U(n)} & & \Ho(\O(2n)\E_n^\mathbf{O}) \ar[d]^{\mathds{L}\res_0^n/\O(2n)} \\
\Ho(n\homog\E_0^\mathbf{U}) \ar[rr]_{\mathds{R}c^*\mathrm{red}(-)} & & \Ho((2n)\homog\E_0^\mathbf{O}) 
}
\]
commutes up to natural isomorphism.
\end{cor}
\begin{proof}
Using the same argument as Corollary \ref{other big diagram for r} using Lemma \ref{big dia for c} and the equivalence of the homotopy categories of $\s^\mathbf{U}[\U(n)]$ and $\U(n)\E_n^\mathbf{U}$. 
\end{proof}

By restricting these larger diagrams, we obtain a homotopy category level commutation result for the lower pentagons of Figure \ref{fig 1}.

\begin{lem}
The diagram 
\[
\xymatrix@C+1cm{
\Ho(\O(n)\E_n^\mathbf{O}) \ar[r]^{\mathds{R}r^*} \ar[d]_{\mathds{L}(\res_0^n/\O(n))} & \Ho(\O(n)\E_n^\mathbf{U}) \ar[r]^{\U(n) \wedge_{\O(n)}^\mathds{L} (-)} & \Ho(\U(n)\E_n^\mathbf{U}) \ar[d]^{\mathds{L}(\res_0^n/\U(n)} \\
\Ho(n\homog\E_0^\mathbf{O}) \ar[rr]_{\mathds{R}r^*\mathrm{red}(-)} & & \Ho(n\homog\E_0^\mathbf{U})
}
\]
of derived functors commutes up to natural isomorphism.
\end{lem}
\begin{proof}
By Lemma \ref{lem: spectra and change of group}, Lemma \ref{square 10 commutes} and Lemma \ref{square 9 commutes}, the composite 
\[
\mathds{L}(\res_0^n/\U(n)) \circ (\U(n)_+ \wedge_{\O(n)}^\mathds{L} (-)) \circ \mathds{R}r^*,
\]
is naturally isomorphic to the composite 
\[
\mathds{L}(\res_0^n/\U(n)) \circ \mathds{R}(\alpha_n)^* \circ \mathds{R}r^* \circ (\U(n)_+ \wedge_{\O(n)}^\mathds{L}(-)) \circ  \mathds{R}{(\beta_n)}_!.
\]
The result then follows by Corollary \ref{other big diagram for r}.
\end{proof}

\begin{lem}
The diagram 
\[
\xymatrix@C+1cm{
\Ho(\U(n)\E_n^\mathbf{U}) \ar[r]^{\mathds{R}c^*} \ar[d]_{\mathds{L}(\res_0^n/\U(n))} & \Ho(\U(n)\E_{2n}^\mathbf{O}) \ar[r]^{\O(2n) \wedge_{\U(n)}^\mathds{L} (-)} & \Ho(\O(2n)\E_{2n}^\mathbf{O}) \ar[d]^{\mathds{L}(\res_0^{2n}/\O(2n))} \\
\Ho(n\homog\E_0^\mathbf{U}) \ar[rr]_{\mathds{R}c^*\mathrm{red}(-)} & & \Ho((2n)\homog\E_0^\mathbf{O})
}
\]
of derived functors commutes up to natural isomorphism.
\end{lem}
\begin{proof}
We have to show that 
\[
\mathds{R}c^* \mathrm{red}(-)\circ \mathds{L}(\res_0^n/\U(n)) \cong \mathds{L}(\res_0^{2n}/\O(2n)) \circ (\O(2n) \wedge_{\U(n)}^\mathds{L} (-)) \circ \mathds{R}c^*.
\]
By Lemma \ref{intermediate categories and spectra for c} and Lemma \ref{change of group for c} we can replace (up to natural isomorphism) the composite  $(\O(2n) \wedge_{\U(n)}^\mathds{L} (-)) \circ \mathds{R}c^*$ with the composite 
\[
\mathds{R}(\beta_{2n})^* \circ (\O(2n) \wedge_{U(n)}^\mathds{L} (-)) \circ \mathds{R}r^* \circ \mathds{L}{(\alpha_n)}_!.
\]
Corollary  \ref{other big diagram for c} and the fact that the homotopy categories of $\s^\mathbf{U}[\O(n)]$ and $\U(n)\E_n^\mathbf{U}$ are equivalent yields that the composite
\[
\mathds{L}(\res_0^{2n}/\O(2n)) \circ \mathds{R}(\beta_{2n})^* \circ (\O(2n) \wedge_{\U(n)}^\mathds{L} (-)) \circ \mathds{R}r^* \circ \mathds{L}{(\alpha_n)}_!
\]
is naturally isomorphic to the composite 
\[
\mathds{R}c^*\mathrm{red}(-) \circ \mathds{L}(\res_0^n/\U(n))
\]
and the result follows. 
\end{proof}

\bibliography{references}
\bibliographystyle{plain}
\end{document}